\documentclass[11pt]{article}

\usepackage{amsmath}
\usepackage{amssymb}
\usepackage{amsthm}
\usepackage{mathtools}
\usepackage{parskip}
\usepackage[dvips,letterpaper,margin=1in]{geometry}
\usepackage{bbm}
\usepackage[linktocpage = true,
            colorlinks = true,
            linkcolor = blue,
            urlcolor  = black,
            citecolor = blue,
            anchorcolor = blue]{hyperref}
\usepackage[capitalise]{cleveref}
\usepackage{enumitem}
\usepackage{xparse}
\usepackage{colonequals}

\usepackage{algorithm}
\usepackage{algpseudocode}

\allowdisplaybreaks

\newtheorem{theorem}{Theorem}
\newtheorem{lemma}{Lemma}

\newcommand{\R}{\mathbb{R}}

\newcommand{\N}{\mathbb{N}}

\newcommand{\Normal}{\mathcal{N}}
\newcommand{\GOE}{\mathrm{GOE}}

\DeclareMathOperator{\var}{var}
\DeclareMathOperator{\cov}{cov}
\DeclareMathOperator{\corr}{corr}
\DeclareMathOperator{\disc}{disc}
\DeclareMathOperator{\tr}{tr}
\DeclareMathOperator{\diag}{diag}
\DeclareMathOperator{\rank}{rank}
\DeclareMathOperator{\symvec}{symvec}

\DeclareMathOperator{\row}{row}

\newcommand{\abs}[1]{\left\lvert\ifblank{#1}{\:\cdot\:}{#1}\right\rvert}
\newcommand{\norm}[1]{\left\lVert\ifblank{#1}{\:\cdot\:}{#1}\right\rVert}
\newcommand{\inp}[2]{\langle#1,#2\rangle}
\newcommand{\of}[1]{{\left(#1\right)}}

\let\oldexp\exp
\RenewDocumentCommand{\exp}{g}{%
\IfNoValueTF{#1}{\oldexp}%
{\,{\oldexp}\of{#1}}}

\NewDocumentCommand{\E}{o g}{%
\IfNoValueTF{#2}{\operatorname{\mathbb{E}}\IfValueT{#1}{_{#1}}}%
{\,{\operatorname{\mathbb{E}}\IfValueT{#1}{_{#1}}}\of{#2}}}

\RenewDocumentCommand{\P}{o g}{%
\IfNoValueTF{#2}{\operatorname{\mathbb{P}}\IfValueT{#1}{_{#1}}}%
{\,{\operatorname{\mathbb{P}}\IfValueT{#1}{_{#1}}}\of{#2}}}

\NewDocumentCommand{\ind}{g}{%
\IfNoValueTF{#1}{\operatorname{\mathbb{I}}}%
{\,\mathbb{I}\of{#1}}}

\usepackage{authblk}

\begin{document}
    \title{Asymptotic Bounds and Online Algorithms for Average-Case Matrix Discrepancy}

    \author[1]{Dmitriy Kunisky\thanks{Email: \texttt{kunisky@jhu.edu}. Partially supported by ONR Award N00014-20-1-2335 and a Simons Investigator Award to Daniel Spielman.}}
    \affil[1]{\large Department of Applied Mathematics \& Statistics, Johns Hopkins University}

    \author{Timm Oertel\thanks{Email: \texttt{timm.oertel@fau.de}.}}
    \affil[2]{Department of Data Science, FAU Erlangen-N\"{u}rnberg}

    \author[2]{Nicola Wengiel\thanks{Email: \texttt{nicola.wengiel@fau.de}.}}

    \author[3]{Peiyuan Zhang\thanks{Email: \texttt{peiyuan.zhang@yale.edu}. Partially supported by a Yale University Fund to Amin Karbasi.}}

    \affil[3]{Department of Electrical Engineering, Yale University}

    \date{}

    \maketitle
    \thispagestyle{empty}

    \begin{abstract}
        We study the matrix discrepancy problem in the average-case setting. Given a sequence of $m \times m$ symmetric matrices $A_1,\ldots,A_n$, its discrepancy is defined as the minimal spectral norm over all signed sums $\sum_{i=1}^n x_iA_i$ with $x_1,\ldots,x_n \in \{\pm1\}$. Our contributions are twofold. First, we study the asymptotic discrepancy of random matrices. When the matrices belong to the Gaussian orthogonal ensemble, we provide a sharp characterization of the asymptotic discrepancy and show that the limiting distribution is concentrated around $\Theta(\sqrt{nm}4^{-(1 + o(1))n/m^2})$, under the assumption $m^2 \ll n/\log{n}$. We observe that the trivial bound $O(\sqrt{nm})$ cannot be improved when $n \ll m^2$ and show that this phenomenon occurs for a broad class of random matrices. In the case $n = \Omega(m^2)$, we provide a matching upper bound. Second, we analyse the matrix hyperbolic cosine algorithm, an online algorithm for matrix discrepancy minimization due to Zouzias~(2011), in the average-case setting. We show that the algorithm achieves with high probability a discrepancy of $O(m\log{m})$ for a broad class of random matrices, including Wigner matrices with entries satisfying a hypercontractive inequality and Gaussian Wishart matrices.
    \end{abstract}

    \clearpage

    \tableofcontents
    \thispagestyle{empty}

    \clearpage

    \section{Introduction}
    \pagenumbering{arabic}

    Discrepancy theory is a branch of combinatorics with applications in different fields of mathematics and computer science. For a thorough introduction to the topic, we refer the reader to the books by Chazelle \cite{Chazelle00} and Matoušek \cite{Matousek10}. Here we are mainly concerned with matrix discrepancy, a natural extension of vector discrepancy that has gained a lot of attention in recent years. We begin our exposition by providing a brief overview of vector and matrix discrepancy.

    \subsection{Literature overview}

    The classical vector discrepancy problem can be described as follows. Given a sequence of vectors $v_1,\ldots,v_n \in \R^m$ satisfying $\norm{v_i}_\infty \leq 1$, the goal is to find a signing $x \in \{\pm1\}^n$ that attains the minimal discrepancy
    \begin{equation*}
        \disc(v_1,\ldots,v_n) \colonequals \min_{x \in \{\pm1\}^n} \norm{\sum_{i=1}^n x_iv_i}_\infty,
    \end{equation*}
    where $\norm{y}_\infty \colonequals \max_{j \in [m]} |y_j|$ denotes the maximum norm. A standard union bound argument together with the Chernoff inequality yields an $O(\sqrt{n\log{m}})$ bound \cite{Matousek10}, achieved with high probability by $x_i \in \{\pm 1\}$ drawn independently and uniformly at random. In his seminal work, Spencer \cite{Spencer85} improved this bound to $O(\sqrt{n\log(2m/n)})$ when $n \leq m$. The case $n > m$ can be reduced to the case $n = m$ using an iterated rounding technique \cite{Beck81}, which leads to a general bound of the form
    \begin{equation}
        \label{eq:SpencerResult}
        \disc(v_1,\ldots,v_n) \leq C
        \begin{cases}
            \sqrt{n\log(2m/n)} & \text{if}~n \leq m, \\
            \sqrt{m} & \text{if}~n \geq m,
        \end{cases}
    \end{equation}
    for some constant $C > 0$. This bound is known to be tight up to the value of $C$ \cite{Alon16}. The above setting where the vectors are restricted to have $\ell^{\infty}$-norm at most one is only one of several interesting settings; other restrictions include sparse vectors, corresponding to the Beck-Fiala conjecture, and vectors with Euclidean norm at most one, corresponding to the Koml\'{o}s conjecture. Both conjectures are still open, but Spencer's result motivated the development of several powerful techniques to tackle these problems \cite{Banaszczyk98,Banaszczyk12,Giannopoulos97,Gluskin88}.

    Over the past few years, vector discrepancy has proven to be a useful tool with various applications. For example, it was used by Hoberg and Rothvoss \cite{Hoberg17} to give the best known approximation algorithm for bin packing, by Bansal, Charikar, Krishnaswamy and Li \cite{Bansal14} to improve rounding of linear programs for broadcast scheduling, and by Chandrasekaran and Vempala \cite{Vempala14} to study integer feasibility of random polytopes. Other problems like prefix discrepancy and vector balancing are closely related to vector discrepancy and have also benefited this progress, leading to even more applications \cite{Barany81,Eisenbrand18,Jansen19,Nikolov17,Sevastjanov94}.

    \textbf{Random instances.} While earlier results in discrepancy theory were mainly concerned with worst-case instances, the study of average-case instances has lately received more attention. A few works have also considered a smoothed analysis of discrepancy problems, where a small random pertubation is applied to the input \cite{Bansal22b,Bansal22c}. Here we focus on the purely probabilistic setting, in which essentially two random models have been considered.

    One line of work, initiated by Karmarkar, Karp, Lueker and Odlyzko \cite{Karmarkar86} and continued by Costello \cite{Costello09} and Turner, Meka and Rigollet \cite{Turner20}, studied the asymptotic discrepancy of $m$-dimensional standard Gaussian vectors $v_1,\ldots,v_n$. The work of Costello \cite{Costello09} considered the constant dimension regime $m = O(1)$. Its main result shows that the distribution of a random variable counting the number of low discrepancy solutions tends to a Poisson distribution. In particular, the author concluded that
    \begin{equation}
        \label{eq:CostelloResult}
        \lim_{n \to \infty} \P{\disc(v_1,\ldots,v_n) \leq \gamma\sqrt{\frac{\pi n}{2}}2^{-n/m}} = 1 - \exp{-2\gamma^m}
    \end{equation}
    for any $\gamma > 0$. The work of Turner, Meka and Rigollet \cite{Turner20} addressed the increasing dimension regime $m = \omega(1)$. Using the second moment method, they showed that
    \begin{equation}
        \label{eq:TurnerResult}
        \lim_{n \to \infty} \P{\beta\sqrt{\frac{\pi n}{2}}2^{-n/m} \leq \disc(v_1,\ldots,v_n) \leq \gamma\sqrt{\frac{\pi n}{2}}2^{-n/m}} = 1
    \end{equation}
    for any $\beta < 1 < \gamma$. In the latter result, it was further assumed that the number of vectors grows asymptotically faster than the dimension, that is, that $m = o(n)$. This is necessary to beat an $O(\sqrt{n})$ bound: for a fixed signing $x \in \{\pm1\}^n$, we have that $\P{\norm{\sum_{i=1}^n x_iv_i}_\infty \leq \delta} \leq O(\delta/\sqrt{n})^m$. If $m = \Omega(n)$, a first moment bound shows that $\P{\disc(v_1,\ldots,v_n) \leq \delta}$ tends to zero for $\delta = n^{1/2-\varepsilon}$. In the regime $n = \Theta(m)$, there is a much more detailed characterization of the typical discrepancy as a function of the aspect ratio $\lim_{n \to \infty} n/m$. This setting, up to changes of notation, goes by the name of the \emph{symmetric binary perceptron} model and is of interest in statistical physics and machine learning \cite{Abbe22,Altschuler23,Aubin19,Perkins21,Sah23}.

    Another line of work studied the discrepancy of $m$-dimensional random vectors $v_1,\ldots,v_n$ with independent Bernoulli distributed entries. After initial results by Hoberg and Rothvoss \cite{Hoberg19}, Franks and Saks \cite{Franks20} and Potukuchi \cite{Potukuchi18}, the final result due to Altschuler and Niles-Weed \cite{Altschuler22} shows that $\disc(v_1,\ldots,v_n) \leq 1$ with high probability for $n = \Omega(m\log{m})$. The hidden constant must be larger than $(2\log{2})^{-1}$, which is the threshold where the expected number of low discrepancy solutions becomes large; this constant cannot be improved without further assumptions.

    \textbf{Online setting.} Spencer's original proof of his result \cite{Spencer85} was based on the non-constructive method of partial coloring. The first polynomial-time algorithm computing a signing achieving the same discrepancy was given by Bansal~\cite{Bansal10}, and led to many further ideas for algorithms finding low-discrepancy signings \cite{Alweiss21,Bansal19a,Bansal19b,Bansal22a,Dadush16,Harshaw24,Liu22,Lovett15,Rothvoss17}. Beyond that, one may ask the more stringent question of whether \emph{online} algorithms can achieve the same discrepancy. In the online setting, an adversary picks a vector $v_t \in [-1,1]^m$ at each time $t = 1,\ldots,n$, and we must choose a sign $x_t \in \{-1,1\}$ irrevocably without knowledge of the upcoming vectors (i.e., of $v_{t^{\prime}}$ for $t^{\prime} > t$). The goal is to keep the discrepancy $\norm{\sum_{i=1}^n x_iv_i}_\infty$, or the prefix discrepancy $\max_{t=1,\ldots,n} \norm{\sum_{i=1}^t x_iv_i}_\infty$, as small as possible. The naive algorithm that picks $x_t \in \{-1,1\}$ at random already achieves an optimal discrepancy of $O(\sqrt{n\log{m}})$ with high probability, as shown by Spencer~\cite{Spencer94}. A lower bound of $\Omega(\sqrt{n})$ follows from a rather simple strategy for the adapative adversary, namely choosing the vector $v_t \in [-1,1]^m$ orthogonal to the current signed sum $\sum_{i=1}^{t-1} x_iv_i$.

    This hopeless situation has led to a shift of attention to other variations in which the power of the adversary is restricted. Here we focus on the stochastic model, where the vectors $v_1,\ldots,v_n$ are drawn independently from some distribution that is known to the online algorithm. For $v_1,\ldots,v_n$ uniformly drawn from $\{-1,1\}^m$, Bansal and Spencer \cite{Bansal20a} presented an online algorithm that achieves with high probability a discrepancy of $O(\sqrt{m})$ and a prefix discrepancy of $O(\sqrt{m}\log{n})$. This result is optimal up to constants when $n = \Theta(m)$, as demonstrated by Gamarnik, Kızıldağ, Perkins and Xu \cite{Gamarnik23}. For general distributions supported on $[-1,1]^m$, Bansal, Jiang, Singla and Sinha \cite{Bansal20b} showed a high probability bound of $O(m^2\log(mn))$ on the prefix discrepancy, which was improved by Bansal, Jiang, Meka, Singla and Sinha \cite{Bansal21} to $O(\sqrt{m}\log^4(mn))$. Compared with the bounds on the true discrepancy for average-case instances, these results show that online algorithms can achieve optimal or nearly optimal discrepancy for small $n = \Theta(m)$. However, these algorithms do not capture the decaying discrepancy once $n \gg m$. One intuition for this is that the cancellations possible in a long stream of vectors involve vectors far apart in the stream, but an online algorithm cannot backtrack and change the signs of vectors it has seen already to achieve these cancellations.

    Another assumption that may be viewed as lying in between the adaptive adversary and the stochastic model is the \emph{oblivious} adversary. Here the adversary is forced to fix the vectors in advance, while the online algorithm can use randomized strategies. For vectors $v_1,\ldots,v_n$ satisfying $\norm{v_i}_2 \leq 1$, Alweiss, Liu and Sawhney \cite{Alweiss21} showed that a simple self-balancing random walk can find signs so that all partial sums are $O(\sqrt{\log(mn)})$-subgaussian, and in particular which achieves a prefix discrepancy of $O(\log(mn))$ against an oblivious adversary with high probability. Kulkarni, Reis and Rothvoss \cite{Kulkarni24} improved this result by showing the existence of an online algorithm that decides signs so that all partial sums are $10$-subgaussian, and gives an optimal $O(\sqrt{\log{n}})$ discrepancy against an oblivious adversary.

    \textbf{Matrix discrepancy.} A natural generalization of vectors to matrices leads to the following problem. Given symmetric matrices $A_1,\ldots,A_n \in \R^{m \times m}$ satisfying $\norm{A_i} \leq 1$, the goal is to find a signing $x \in \{\pm1\}^n$ with minimal discrepancy
    \begin{equation*}
        \disc(A_1,\ldots,A_n) \colonequals \min_{x \in \{\pm1\}^n} \norm{\sum_{i=1}^n x_iA_i},
    \end{equation*}
    where $\norm{}$ denotes the spectral norm. The spectral norm of a matrix $A$ is defined as $\norm{A} \colonequals \max_{\norm{x}_2 = 1} \norm{Ax}_2$ and corresponds to the largest absolute value of its eigenvalues when $A$ is symmetric. As in the vector case, an $O(\sqrt{n\log{m}})$ bound can be obtained by a union bound and the matrix Chernoff inequality \cite{Tropp12}. It is conjectured that a matrix version of Spencer's discrepancy bound \cite{Spencer85} holds, that is, $\disc(A_1,\ldots,A_n) = O(\sqrt{m\log(2m/n)})$ when $n \leq m$. However, in the matrix setting, the iterated rounding technique only applies when $n > m^2$, and therefore we expect a general bound of the more complicated form
    \begin{equation}
        \label{eq:MatrixSpencerConjecture}
        \disc(A_1,\ldots,A_n) \leq C
        \begin{cases}
            \sqrt{n\log(2m/n)} & \text{if}~n \leq m, \\
            \sqrt{n} & \text{if}~m \leq n \leq m^2, \\
            m & \text{if}~n \geq m^2
        \end{cases}
    \end{equation}
    for some constant $C > 0$. Although the conjecture is still open, there has been considerable progress. Hopkins, Raghavendra and Shetty \cite{Hopkins22} proved the conjecture under the additional assumption that the matrices have Frobenius norm at most $O(n^{1/4})$. Dadush, Jiang and Reis \cite{Dadush22} established the conjecture for block-diagonal matrices with block size less than $O(n/m)$. More recently Bansal, Jiang and Meka \cite{Bansal23} proved the conjecture when the matrices have rank at most $O(n/\log^3{n})$.

    As noted by Bansal, Jiang and Meka \cite{Bansal23}, matrix discrepancy is closely related to random matrix theory, in particular to a detailed understanding of the random spectral norm of $\sum_{i=1}^n x_iA_i$ for a random signing $x \in \{\pm1\}^n$. Aside from being a theoretically exciting question, several applications of matrix discrepancy are known. For example, matrix discrepancy played an important role in the resolution of the Kadison-Singer problem due to Marcus, Spielman and Srivastava \cite{Marcus15}, was applied by Reis and Rothvoss \cite{Reis20} in the context of graph sparsification, and was shown by Hopkins, Raghavendra and Shetty \cite{Hopkins22} to have interesting connections with quantum communication.

    \subsection{Main results}

    In this work, alongside Maillard \cite{Maillard24}, we initiate the study of the matrix discrepancy problem for random instances. Motivated in part by the open matrix Spencer conjecture, our goal is to understand to what extent vector discrepancy results transfer to the matrix setting. Here we mainly focus on asymptotic bounds and online algorithms in the average-case setting.

    \textbf{Notation.} The asymptotic notations $O,\Omega,\Theta,o,\omega$ have their standard meaning and should be understood in the limit $n \to \infty$. We use the shorthand notation $a_n \lesssim b_n$ for $a_n \leq (1 + o(1))b_n$ and $a_n \approx b_n$ for $a_n = (1 + o(1))b_n$. The subscript of sequences is usually omitted. For example, when writing $a \approx 1$, this means that $a = a_n$ is a sequence that tends to one as $n$ goes to infinity. We use the Vinogradov notation $a \ll b$ to denote that $a \leq \varepsilon b$ for some sufficiently small $\varepsilon > 0$. For a natural number $n \in \N$, we denote $[n] \colonequals \{1, \dots, n\}$. We write $\delta_{ab}$ for the Kronecker delta, equal to 1 if $a = b$ and 0 otherwise.

    For a matrix $A \in \R^{m \times m}$, we denote its Froebnius norm by $\norm{A}_F \colonequals \sqrt{\tr(A^TA)}$ and its trace norm by $\norm{A}_* \colonequals \tr(\sqrt{A^TA})$. For two matrices $A,B \in \R^{m \times m}$, we denote its Frobenius inner product by $\inp{A}{B} \colonequals \tr(A^TB)$. For a symmetric matrix $A \in \R^{m \times m}$, we write $\lambda_1(A) \leq \ldots \leq \lambda_m(A)$ for its ordered eigenvalues and let $\lambda_{\min}(A) \colonequals \lambda_1(A)$, $\lambda_{\max}(A) \colonequals \lambda_m(A)$ denote its smallest and largest eigenvalues, respectively. Then $\norm{A} = \max_{j \in [m]} \abs{\lambda_j(A)}$ and $\norm{A}_F = (\sum_{j \in [m]} \lambda_j(A)^2)^{1/2}$. We define the symmetric vectorization of $A$ as the $m(m+1)/2$-dimensional vector with entries $a_{ii}$ for $1 \leq i \leq n$ and $\sqrt{2}a_{ij}$ for $1 \leq i < j \leq m$, and denote it by $\symvec(A)$. Note that $\inp{A}{B} = \inp{\symvec(A)}{\symvec(B)}$ and $\norm{A}_F = \norm{\symvec(A)}_2$.

    We denote the indicator function of an event $E$ by $\ind{E}$. The expectation (also called mean) and variance of a random variable $X$ are denoted by $\E{X}$ and $\var(X) \colonequals \E{(X - \E{X})^2}$, respectively. If there is no risk of confusion, we drop the brackets in the notation $\E{X}$ and simply write $\E X$. The covariance and correlation between two random variables $X$ and $Y$ are denoted by $\cov(X,Y) \colonequals \E (X - \E{X})(Y - \E{Y})$ and $\corr(X,Y) \colonequals \cov(X,Y)(\var(X)\var(Y))^{-1/2}$, respectively. For a $k$-dimensional random vector $X$, we denote its covariance matrix with entries $\cov(X_i,X_j)$ for $1 \leq i,j \leq k$ by $\cov(X)$.

    We write $\Normal(\mu, \sigma^2)$ for the scalar Gaussian distribution with mean $\mu$ and variance $\sigma^2$, and $\Normal(\mu, \Sigma)$ for the multivariate Gaussian distribution with mean vector $\mu$ and covariance matrix $\Sigma$.
    We write $\GOE(m)$ for the law of an $m \times m$ symmetric random matrix $X$ with $X_{ij} = X_{ji} \sim \Normal(0, 1 + \delta_{ij})$ independently for all $i \leq j$.
    This is the \emph{Gaussian orthogonal ensemble (GOE)}.
    For a sequence of events $E_n$ (possibly over different probability spaces depending on $n$), we say that the sequence holds \emph{with high probability} if $\P{E_n} \to 1$.

    \textbf{Asymptotic results.} The first part of our work is concerned with the asymptotic discrepancy of random matrices. Many interesting random matrix models (also called random matrix \emph{ensembles}) could be considered here. We restrict ourselves to symmetric square matrices $X = (X_{ij})_{1 \leq i,j \leq m}$. A basic model for symmetric random matrices is the \emph{Wigner matrix ensemble}, in which the upper diagonal entries $(X_{ij})_{i \leq j}$ are jointly independent, and $X_{ji} \colonequals X_{ij}$ below the diagonal \cite{Tao12}.

    One important example is $X \sim \GOE(m)$, as defined above.
    Apart from being one of the most commonly studied models in random matrix theory due to its orthogonal symmetry \cite{Anderson10,Tao12}, the GOE can also be viewed as a natural generalization of Gaussian random vectors to symmetric matrices. We consider the question of whether asymptotic results analogous to the analysis of Gaussian vectors of \eqref{eq:CostelloResult} and \eqref{eq:TurnerResult} can be established in the matrix setting. In view of the fact that the spectral norm of GOE matrices is concentrated in the range $[2\sqrt{m} - O(m^{-1/6}),2\sqrt{m} + O(m^{-1/6})]$ according to the Tracy-Widom limit theorem \cite{Soshnikov99}, one would expect that the asymptotic discrepancy of GOE matrices is $\Theta(\sqrt{nm}4^{-n/m^2})$ with high probability, where $4^{-n/m^2} = 2^{-2n/m^2}$ takes into account that the dimension of $m \times m$ symmetric matrices is $m(m + 1) / 2 \approx m^2/2$. The following result confirms this intuitive guess.

    \begin{theorem}
        \label{thm:GaussianDiscrepancy}
        Let $A_1,\ldots,A_n \sim \GOE(m)$ independently for some $m = m(n)$.
        \begin{enumerate}[label=(\alph*)]
            \item Assume that $\omega(1) = m^2 \ll n/\log{n}$. Then, for any constants $\beta < 1 < \gamma$,
            \begin{equation*}
                \lim_{n \to \infty} \P{\beta\frac{2}{e^{3/4}}\sqrt{nm}4^{-\xi n/m^2} < \disc(A_1,\ldots,A_n) \leq \gamma\frac{2}{e^{3/4}}\sqrt{nm}4^{-\xi n/m^2}} = 1,
            \end{equation*}
            where $\xi = \xi(n)$ is a sequence with $\xi(n) \to 1$ as $n \to \infty$.
            \label{thm:GaussianDiscrepancy1}
            \item Assume that $m = O(1)$. Then, for any constant $\beta < 1$,
            \begin{equation*}
                \lim_{n \to \infty} \P{\disc(A_1,\ldots,A_n) > \beta\frac{2}{e^{3/4}}\sqrt{nm}4^{-\xi n/m^2}} \geq 1 - \beta^{\xi^{-1}m^2/2},
            \end{equation*}
            and for any constant $\gamma > 1$,
            \begin{equation*}
                \lim_{n \to \infty} \P{\disc(A_1,\ldots,A_n) \leq \gamma\frac{2}{e^{3/4}}\sqrt{nm}4^{-\xi n/m^2}} \geq \frac{1}{1 + 2\gamma^{-\xi^{-1} m^2/2}},
            \end{equation*}
            where $\xi = \xi(n)$ is a sequence under the same assumption as in Part \ref{thm:GaussianDiscrepancy1}.
            \label{thm:GaussianDiscrepancy2}
        \end{enumerate}
    \end{theorem}

    The theorem demonstrates that the (random) value of $\disc(A_1,\ldots,A_n)$ concentrates around the value $2e^{-3/4}\sqrt{nm}4^{-\xi n/m^2}$. Notice that this quantity tends to zero if $m^2 \ll n/\log{n}$. For comparison, the spectral norm of an arbitrary signed sum $\sum_{i=1}^n x_iA_i$ for some fixed $x_i$ is $\Theta(\sqrt{nm})$ with high probability. This highlights that drastic cancellations are possible if the number of matrices $n$ grows fast enough. We further remark that the assumption $m^2 \ll n/\log{n}$ in Part~\ref{thm:GaussianDiscrepancy1} is only required for the upper bound. We leave open the question of whether the upper bound also holds in the intermediate regime $\Omega(n/\log{n}) = m^2 = o(n)$. That the assumption $m^2 = o(n)$ is necessary to improve on the trivial bound $\Theta(\sqrt{nm})$ can be justified as follows: for a fixed signing $x \in \{\pm1\}^n$, we have that $\P{\norm{\sum_{i=1}^n x_iA_i}_\infty \leq \delta} \leq O(\delta/\sqrt{nm})^{(1 + o(1))m^2/2}$ by \cref{lem:SmallNormProbability}. Then, if $m^2 = \Omega(n)$, a first moment bound shows that $\P{\disc(A_1,\ldots,A_n) \leq \delta}$ tends to zero for $\delta = n^{1/2-\varepsilon}m^{1/2}$.

    One may interpret \cref{thm:GaussianDiscrepancy} as a matrix version of the asymptotic discrepancy results for Gaussian vectors discussed above, where Part~\ref{thm:GaussianDiscrepancy1} is the analog of result~\eqref{eq:TurnerResult} in the increasing dimension regime $m = \omega(1)$ and Part~\ref{thm:GaussianDiscrepancy2} is the analog of result~\eqref{eq:CostelloResult} in the constant dimension regime $m = O(1)$. However, unlike Costello's result \eqref{eq:CostelloResult}, we are not able to obtain the exact limiting distribution in the constant dimension regime. This is because our techniques, essentially an application of the second moment method, do not exploit higher moments.

    We organize the proof of \cref{thm:GaussianDiscrepancy} as follows: in \cref{sec:Preliminaries} we begin with some preliminary results on GOE matrices, and then give the proof based on the second moment method in \cref{sec:GaussianDiscrepancy}.

    More detailed information on the regime $n = \Theta(m^2)$ can be found in the parallel work of Maillard \cite{Maillard24}, which establishes a satisfiability transition in terms of the aspect ratio $\tau = \lim_{n \to \infty} n/m^2$. For the interesting regime $\kappa < 2$, the author identified functions $\tau_1(\kappa) ,\tau_2(\kappa)$ such that
    \begin{equation}
        \label{eq:MaillardLowerBound}
        \lim_{n \to \infty} \P{\disc(A_1,\ldots,A_n) > \kappa\sqrt{nm}} = 1
    \end{equation}
    if $\tau < \tau_1(\kappa)$, and
    \begin{equation}
        \label{eq:MaillardUpperBound}
        \lim_{n \to \infty} \P{\disc(A_1,\ldots,A_n) \leq \kappa\sqrt{nm}} = 1
    \end{equation}
    if $\tau > \tau_2(\kappa)$. The existence of a sharp threshold function $\tau_c$, with the property that $\disc(A_1,\ldots,A_n)$ is concentrated around the single value $\kappa\sqrt{nm}$ when $\tau = \tau_c(\kappa)$, follows from Theorem 7 of Altschuler \cite{Altschuler24}, and the above bounds locate it in the interval $\tau_1(\kappa) \leq \tau_c(\kappa) \leq \tau_2(\kappa)$.

    The lower bound \eqref{eq:MaillardLowerBound} illustrates once again that $\disc(A_1,\ldots,A_n) = \Omega(\sqrt{nm})$ when $n \ll m^2$.
    Our next result shows that this behavior in fact holds for a broad class of random matrices, under two conditions. The first condition is a uniform subexponential tail inequality for the linear forms $\inp{X}{Y}$ with respect to symmetric matrices $Y \in \R^{m \times m}$. To state this condition formally, we use the concept of the $\psi_r$-norm. Recall that the $\psi_r$-norm of a random variable $X$ is defined as
    \begin{equation*}
        \norm{X}_{\psi_r} \colonequals \inf\{C > 0 : \E\exp{\abs{X}^r/C^r} \leq 2\}.
    \end{equation*}
    A random variable with finite $\psi_1$-norm is called \emph{subexponential} and a random variable with finite $\psi_2$-norm is called \emph{subgaussian}; as is well-known, bounds on these norms are equivalent to subexponential and subgaussian tail bounds, respectively.
    The concept of the $\psi_r$-norm can be extended to $m$-dimensional random vectors $X$ by considering scalar projections
    \begin{equation*}
        \norm{X}_{\psi_r} \colonequals \sup_{\norm{Y}_2 = 1} \norm{\inp{X}{Y}}_{\psi_r}.
    \end{equation*}
    Likewise, we generalize the $\psi_r$-norm to $m \times m$ random matrices $X$ by
    \begin{equation*}
        \norm{X}_{\psi_r} \colonequals \norm{\symvec(X)}_{\psi_r} = \sup_{\norm{Y}_F = 1} \norm{\inp{X}{Y}}_{\psi_r}.
    \end{equation*}
    Our first condition below is a uniform bound on the $\psi_1$-norm.
    The second condition is a weaker concentration inequality for the Frobenius norm,
    \begin{equation}
        \label{eq:NormConcentration}
        \P{\abs{\frac{1}{m^2}\norm{X}_F^2 - 1} \geq \frac{1}{2}} = o\of{\frac{1}{m^2}}.
    \end{equation}

    \begin{theorem}
        \label{thm:GeneralDiscrepancy}
        Assume that $A_1,\ldots,A_n$ is a sequence of independent centered $m \times m$ random symmetric matrices, for some $m = m(n)$, that have $\|A_i\|_{\psi_1} \leq \psi$ for all $i \in [n]$ and satisfy Condition~\eqref{eq:NormConcentration}. Then, there exist constants $C_1,C_2 > 0$ such that for $n \ll m^2$,
        \begin{equation*}
            \lim_{n \to \infty} \P{C_1\sqrt{nm} \leq \disc(A_1,\ldots,A_n) \leq C_2m^{3/2}} = 1.
        \end{equation*}
        In particular, when $n = \Theta(m^2)$, we have that $\disc(A_1,\ldots,A_n) = \Theta(m^{3/2})$ with high probability as $n \to \infty$.
    \end{theorem}

    The proof of \cref{thm:GeneralDiscrepancy} can be found in \cref{sec:GeneralDiscrepancy}.
    We next discuss two applications of \cref{thm:GeneralDiscrepancy} to specific matrix ensembles. We first consider Wigner matrices of the form $A = X + X^T$, where $X$ is an $m \times m$ random matrix with independent and identically distributed subgaussian entries. It follows from Lemma 3.4.2 of \cite{Vershynin18} that $\symvec(A)$ is a subgaussian vector with $\psi_2$-norm bounded by a constant $C > 0$ depending only on the entry distribution. In particular, $\norm{\symvec(A)}_{\psi_1} \leq \norm{\symvec(A)}_{\psi_2} \leq C$. Furthermore, Theorem 3.1.1 of \cite{Vershynin18} implies that $\|\symvec(A)\|_F^2$ is concentrated around $m^2$. Since $\norm{A}_{\psi_1} = \norm{\symvec(A)}_{\psi_1}$ and $\norm{A}_F = \norm{\symvec(A)}_2$, a sequence $A_1,\ldots,A_n$ of $n$ independent copies of $A$ meets the requirements of \cref{thm:GeneralDiscrepancy} and the conclusion of the theorem applies. In particular, it applies to the case where $A_1,\ldots,A_n \sim \GOE(m)$.

    Our second application concerns Gaussian Wishart random matrices. We say that $W$ is a \emph{Wishart matrix} with rank $r = r(n) \leq m$ if it is of the form $W = GG^T$ for some $m \times r$ matrix $G$ with independent standard Gaussian entries. Note that $\rank(W) = r$ by construction. Unfortunately, \cref{thm:GeneralDiscrepancy} is not directly applicable in this situation, as Wishart matrices are not centered. We present a workaround in \cref{sec:GeneralDiscrepancy}, which leads to the following result.

    \begin{theorem}
        \label{thm:WishartDiscrepancy}
        Let $W_1,\ldots,W_n$ be a sequence of independent $m \times m$ Wishart matrices with rank $r \leq m$. There exist constants $C_1,C_2 > 0$ such that, for $n \ll m^2$,
        \begin{equation*}
            \lim_{n \to \infty} \P{C_1\sqrt{rnm} \leq \disc(A_1,\ldots,A_n) \leq C_2\sqrt{rm^3}} = 1.
        \end{equation*}
        In particular, when $n = \Theta(m^2)$, we have that $\disc(A_1,\ldots,A_n) = \Theta(\sqrt{rm^3})$ with high probability as $n \to \infty$.
    \end{theorem}

    \textbf{Matrix hyperbolic cosine algorithm.} In the second part of our work, we study the matrix hyperbolic cosine (MHC) algorithm given in \cref{alg:MatrixHyperbolicCosine}, which is a matrix version of an online algorithm for vector discrepancy that was introduced by Spencer \cite{Spencer77} and studied in the average-case setting by Bansal and Spencer \cite{Bansal20a}. In short, the matrix version of the algorithm maintains a single matrix $M$ of the current signed sum and in each step picks a sign in order to minimize the potential function $\tr\cosh(\alpha M)$, where $\alpha$ is an appropriately chosen parameter. The algorithm for matrices was first introduced and studied by Zouzias \cite{Zouzias12}.

    \begin{algorithm}
        \caption{Matrix hyperbolic cosine (MHC) algorithm}
        \label{alg:MatrixHyperbolicCosine}
        \begin{algorithmic}
            \Require{Sequence of matrices $A_1,\ldots,A_n \in \R^{m \times m}$ and parameter $\alpha > 0$.}
            \Ensure{Sequence of signs $x_1,\ldots,x_n \in \{\pm1\}$.}
            \State{Initialize $M \gets 0$.}
            \For{$t = 1,\ldots,n$}
                \State{Choose $x_t \in \{\pm1\}$ to minimize $\tr\cosh(\alpha(M + x_tA_t))$.}
                \State{Update $M \gets M + x_tA_t$.}
            \EndFor
            \State{\Return $x_1,\ldots,x_n$.}
        \end{algorithmic}
    \end{algorithm}

    The analysis of Zouzias \cite{Zouzias12} gives a bound of $O(\sqrt{n\log{m}})$ on the discrepancy achieved by the MHC algorithm. This is much larger than the bound suggested by the matrix Spencer conjecture when $n$ is large, and indeed is just the same as the non-commutative Khintchine inequality implies for random signs \cite{Lust91}. In fact, the purpose of Zouzias in analyzing the MHC algorithm was precisely to achieve the performance of random signs with a deterministic algorithm. Here we are instead interested in a sharper characterization of the algorithm's performance for random inputs. We introduce two conditions on random matrix distributions that allow us to establish upper bounds on the discrepancy achieved by the MHC algorithm. In these conditions, we use an additional parameter $r$, which should be thought of as the rank, although it can be chosen arbitrarily in order to fulfill the conditions.

    The first condition is an anti-concentration inequality, analogous to such conditions for random vectors. We say that an $m \times m$ random symmetric matrix $A$ satisfies the \emph{matrix anti-concentration inequality} with parameter $\eta > 0$ if
    \begin{equation}
        \label{eq:MatrixAntiConcentration}
        \E\abs{\inp{X}{A}} \geq \eta\sqrt{\frac{r}{m^3}}\norm{X}_*
    \end{equation}
    for all symmetric matrices $X \in \R^{m \times m}$. We will only ever prove a matrix anti-concentration condition through a stronger Khintchine-like inequality \eqref{eq:KhintchineAntiConcentration}. We state our main result in terms of this weaker condition to draw a parallel with \cite{Bansal20b} where such a condition was used in the vector case to prove weaker discrepancy results over a broader range of distributions.

    The second condition is a quantitative isotropy condition for the row space. The row space of a matrix $A$ is the span of its row vectors and denoted by $\row(A)$. We say that an $m \times m$ random matrix $A$ is unbiased with parameter $\theta > 0$ if
    \begin{equation}
        \label{eq:Unbiasedness}
        \norm{\E P_{\row(A)}} \geq \theta\frac{r}{m},
    \end{equation}
    where $P_V$ denotes the matrix of the orthogonal projection onto a subspace $V$. For intuition, note that the unbiasedness condition is only non-trivial when $r \ll m$ so that $A$ is low-rank. If $V$ is a uniformly distributed random subspace of dimension $r$, then by symmetry we have $\E P_V = \frac{r}{m}I_m$. The unbiasedness condition therefore says that the distribution of $\row(A)$ is quantitatively close to the uniform distribution on $r$-dimensional subspaces.

    \begin{theorem}
        \label{thm:GeneralAnalysis}
        Assume that $A$ is an $m \times m$ random symmetric matrix with $\norm{A} \leq 1$ that satisfies Conditions \eqref{eq:MatrixAntiConcentration} and \eqref{eq:Unbiasedness} for parameters $\eta,\theta > 0$. Let $A_1,\ldots,A_n$ be a sequence of $n$ independent copies of $A$, and let $x_1,\ldots,x_n \in \{\pm1\}$ denote the signs produced by \cref{alg:MatrixHyperbolicCosine} when run with parameter $\alpha \ll (rm)^{-1/2}$. Then,
        \begin{enumerate}[label=(\alph*)]
            \item with probability at least $1 - n^{-1}$,
            \begin{equation*}
                \max_{t=1,\ldots,n} \norm{\sum_{i=1}^t x_iA_i} \leq O(\sqrt{rm}\log{n}),
            \end{equation*}
            \label{thm:GeneralAnalysis1}
            \item with probability at least $1 - m^{-1}$,
            \begin{equation*}
                \norm{\sum_{i=1}^n x_iA_i} \leq O(\sqrt{rm}\log{m}),
            \end{equation*}
            \label{thm:GeneralAnalysis2}
        \end{enumerate}
        where the implicit constants depend only on the parameters $\eta$ and $\theta$.
    \end{theorem}

    Note that Part~\ref{thm:GeneralAnalysis1} of \cref{thm:GeneralAnalysis} provides a bound on the prefix discrepancy (that is, the maximal spectral norm of all partial sums), while Part~\ref{thm:GeneralAnalysis2} only bounds the discrepancy. The proof of \cref{thm:GeneralAnalysis} is given in \cref{sec:GeneralAnalysis}. Applications of \cref{thm:GeneralAnalysis} to specific matrix ensembles are discussed in \cref{sec:Applications} and include GOE matrices, general Wigner matrices with subgaussian entries, and Wishart matrices.

    \section{Preliminaries}
    \label{sec:Preliminaries}

    In this section, we collect relevant results on the Gaussian orthogonal ensemble. We begin with a quick reminder about the multivariate normal distribution. The density of $X = (X_1,\ldots,X_k) \sim \Normal(\mu, \Sigma)$ is given by
    \begin{equation}
        \label{eq:MultivariateNormalDistribution}
        x \mapsto \frac{1}{\sqrt{(2\pi)^k\det(\Sigma)}} \exp{-\frac{(x - \mu)^T\Sigma^{-1}(x - \mu)}{2}}, \quad x \in \R^k.
    \end{equation}
    In particular, the density of two jointly normal random variables $X,Y$ with mean zero, variance $\sigma^2 > 0$ and correlation $\rho \in (-1,1)$ is given by
    \begin{equation}
        \label{eq:BivariateNormalDistribution}
        (x,y) \mapsto \frac{1}{2\pi\sigma^2\sqrt{1 - \rho^2}} \exp{-\frac{x^2 - 2\rho xy + y^2}{2\sigma^2(1 - \rho^2)}}, \quad x,y \in \R.
    \end{equation}

    \textbf{Gaussian orthogonal ensemble.}
    Recall that $X \sim \GOE(m)$ if $X_{ij} = X_{ji} \sim \Normal(0, 1 + \delta_{ij})$ independently for $i \leq j$.
    GOE matrices are so called because they are invariant under orthogonal transformations, that is, if $Q$ is orthogonal, then $QXQ^T$ has the same distribution as $X$.
    Another remarkable fact is that the density of the eigenvalues of a GOE matrix can be written in closed form.
    We will use this to determine the asymptotics of the probability that a GOE matrix has small spectral norm.
    Using the expression for the normal density in \eqref{eq:MultivariateNormalDistribution}, the following is easy to derive.

    \begin{lemma}[Equation 2.5.1 in \cite{Anderson10}]
        \label{lem:GaussianMatrix1}
        The probability density function of the measure $\GOE(m)$ is given by
        \begin{equation*}
            X \mapsto K_m \exp{-\frac{1}{4}\tr(X^2)}
        \end{equation*}
        on the space of $m \times m$ symmetric matrices, where the normalization constant $K_m$ is defined as
        \begin{equation}
            \label{eq:NormalizationConstant1}
            K_m \colonequals 2^{-m/2}(2\pi)^{-m(m+1)/4}.
        \end{equation}
    \end{lemma}

    Since the trace is invariant under orthogonal transformations, the orthogonal invariance of the GOE follows immediately.
    We introduce some further notation. For $\lambda \in \R^m$ we define the \emph{Vandermonde determinant} by
    \begin{equation*}
        \Delta(\lambda) \colonequals \prod_{1 \leq i < j \leq m} (\lambda_j - \lambda_i),
    \end{equation*}
    and for $z > 0$ we define the Gamma function by
    \begin{equation*}
        \Gamma(z) \colonequals \int_0^\infty x^{z-1} e^{-x} dx.
    \end{equation*}
    The next lemma describes the joint distribution of the eigenvalues in the Gaussian orthogonal ensemble.

    \begin{lemma}[Theorem 2.5.2 in \cite{Anderson10}]
        \label{lem:GaussianEigenvalues1}
        Let $X \sim \GOE(m)$. The joint probability density function of the ordered eigenvalues $\lambda_1(X) \leq \ldots \leq \lambda_m(X)$ is given by
        \begin{equation*}
            \lambda \mapsto C_m \exp{-\frac{1}{4}\norm{\lambda}_2^2} \Delta(\lambda)
        \end{equation*}
        on the Weyl chamber $\R^m_\geq \colonequals \{\lambda \in \R^m : \lambda_1 \leq \ldots \leq \lambda_m\}$, where the constant $C_m$ is defined as
        \begin{equation}
            \label{eq:NormalizationConstant2}
            C_m \colonequals 2^{-m(m+3)/4} \prod_{i=1}^m \frac{1}{\Gamma(i/2)}.
        \end{equation}
    \end{lemma}

    For further details on GOE matrices, we refer to the book by Anderson, Guionnet and Zeitouni \cite{Anderson10}. We now turn to studying correlated GOE matrices, which arise naturally in our approach of using the second moment method.

    \textbf{Correlated GOE matrices.} We say that two GOE matrices $X$ and $Y$ have a correlation of $\rho \in [-1,1]$ if $\corr(X_{ij},Y_{ij}) = \rho$ for all $i \leq j$.
    We always assume that the entries of $X$ and $Y$ are jointly normally distributed (as will always be the case in our applications since the pairs of GOE matrices we consider will arise from linear combinations of independent GOE matrices).
    We have the following versions of the above results for the joint distribution of $(X, Y)$.

    \begin{lemma}
        \label{lem:GaussianMatrix2}
        The joint probability density function of two $m \times m$ GOE matrices with correlation $\rho \in (-1,1)$ is given by
        \begin{equation*}
            (X,Y) \mapsto K_{m}^2 (1 - \rho^2)^{-m(m+1)/4} \exp{-\frac{\tr(X^2 - 2\rho XY + Y^2)}{4(1 - \rho^2)}}
        \end{equation*}
        on the space of pairs of $m \times m$ symmetric matrices, where the constant $K_m$ is defined as in \eqref{eq:NormalizationConstant1}.
    \end{lemma}

    \begin{lemma}
        \label{lem:GaussianEigenvalues2}
        Let $X,Y$ be two $m \times m$ GOE matrices with correlation $\rho \in (-1,1)$. The joint probability density function of their ordered eigenvalues $\lambda_1(X) \leq \ldots \leq \lambda_m(X)$ and $\mu_1(X) \leq \ldots \leq \mu_m(X)$ is bounded above by
        \begin{equation}
            \label{eq:GaussianEigenvalues}
            (\lambda,\mu) \mapsto C_m^2 (1 - \rho^2)^{-m(m+1)/4} \exp{-\frac{\norm{\lambda}_2^2 - 2\rho\inp{\lambda}{\mu} + \norm{\mu}_2^2}{4(1 - \rho^2)}} \Delta(\lambda)\Delta(\mu)
        \end{equation}
        on $\R^m_\geq \times \R^m_\geq$, where the constant $C_m$ is defined as in \eqref{eq:NormalizationConstant2}.
    \end{lemma}

    The proofs of \cref{lem:GaussianMatrix2} and \cref{lem:GaussianEigenvalues2} can be found in \cref{sec:Appendix}; we refrain from a full proof of \cref{lem:GaussianEigenvalues2}, since it would make the paper significantly longer, but for the sake of illustration include a proof of the two-dimensional case.

    \textbf{Small norm probability.} Studying the discrepancy of GOE matrices naturally leads to the task of quantifying the probability $\P{\norm{X} \leq \delta}$ for $X \sim \GOE(m)$. For $\delta = \Omega(\sqrt{m})$, reasonable bounds can be obtained via standard concentration results for the spectral norm. But in our case, $\delta/\sqrt{m}$ will tend to zero and a suitable bound must be obtained by other means.

    \begin{lemma}
        \label{lem:SmallNormProbability}
        Let $X \sim \GOE(m)$. For $\delta = o(\sqrt{m})$, we have
        \begin{equation*}
            \P{\norm{X} \leq \delta} = \left(\frac{e^{3/4}}{2\sqrt{m}}\delta\right)^{(1 + o(1))m^2/2}.
        \end{equation*}
    \end{lemma}

    The proof of \cref{lem:SmallNormProbability} appears in \cref{sec:Appendix}. We remark that \cref{lem:SmallNormProbability} implies the existence of a sequence $\xi = \xi(m)$ that tends to one as $m$ goes to infinity such that
    \begin{equation}
        \label{eq:SmallNormProbability}
        \P{\norm{X} \leq \delta} = \left(\frac{e^{3/4}}{2\sqrt{m}}\delta\right)^{\xi^{-1}m^2/2},
    \end{equation}
    where the use of $\xi^{-1}$ makes the further exposition more convenient. From the proof of \cref{lem:SmallNormProbability} it can be concluded that $\xi$ converges faster to one than $1 + O(m^{-1})$. A careful analysis of the proof would yield even more accurate estimates on $\xi$, but we do not pursue this direction. Furthermore, we would like to draw attention to Proposition 2.1 of \cite{Maillard24}, where the small norm probability of GOE matrices in the regime $\delta/\sqrt{m} = \kappa$ for some constant $\kappa > 0$ is investigated. Using the large deviations principle of Ben Arous and Guionnet \cite{BenArous97}, the author established the asymptotic
    \begin{equation*}
        \lim_{m \to \infty} \frac{1}{m^2} \P{\norm{X} \leq \delta} =
        \begin{cases}
            \frac{\kappa^4}{128} - \frac{\kappa^2}{8} + \frac{1}{2}\log{\frac{\kappa}{2}} + \frac{3}{8} & \text{if}~\kappa \leq 2, \\
            0 & \text{if}~\kappa > 2.
        \end{cases}
    \end{equation*}

    \section{Discrepancy of GOE matrices}
    \label{sec:GaussianDiscrepancy}

    In this section, we carry out the proof of \cref{thm:GaussianDiscrepancy} that provides exact bounds on the asymptotic discrepancy of GOE matrices. We first outline our proof strategy. We will study the random variable $S_n(\varepsilon)$ that counts the number of signings with discrepancy at most $\varepsilon > 0$,
    \begin{equation}
        \label{eq:CountingVariable}
        S_n(\varepsilon) \colonequals \sum_{x \in \{\pm1\}^n} \ind{\norm{\sum_{i=1}^n x_iA_i} \leq \varepsilon}.
    \end{equation}
    The events $\disc(A_1,\ldots,A_n) > \varepsilon$ and $\disc(A_1,\ldots,A_n) \leq \varepsilon$ correspond to the events $S_n(\varepsilon) = 0$ and $S_n(\varepsilon) > 0$, respectively. So, proving \cref{thm:GaussianDiscrepancy} boils down to showing that the probability of the event $S_n(\varepsilon) = 0$ has a sharp threshold at the critical value $\frac{2}{e^{3/4}}\sqrt{nm}4^{-\xi n/m^2}$. For this purpose, we use the first and second moment methods.

    \subsection{Lower bound via the first moment method}
    \label{sec:FirstMomentBound}

    The first moment method uses the first moment of a random variable and Markov's inequality to establish an upper bound on the probability of the variable exceeding a certain value. In the next lemma, we calculate the first moment of $S_n(\varepsilon)$.

    \begin{lemma}
        \label{lem:FirstMoment}
        The first moment of $S_n(\varepsilon)$ is given by
        \begin{equation*}
            \E S_n(\varepsilon) = 2^n \P{\norm{X} \leq \frac{\varepsilon}{\sqrt{n}}},
        \end{equation*}
        where $X \sim \GOE(m)$.
    \end{lemma}

    \begin{proof}
        Noting that $X = \frac{1}{\sqrt{n}} \sum_{i=1}^n x_iA_i$ is a GOE matrix for any signing $x \in \{\pm 1\}^n$ and using linearity of expectation yields
        \begin{equation*}
            \E S_n(\varepsilon) = \sum_{x \in \{\pm 1\}^n} \P{\norm{\frac{1}{\sqrt{n}} \sum_{i=1}^n x_iA_i} \leq \frac{\varepsilon}{\sqrt{n}}} = 2^n \P{\norm{X} \leq \frac{\varepsilon}{\sqrt{n}}}.
            \qedhere
        \end{equation*}
    \end{proof}

    Since $S_n(\varepsilon)$ is non-negative and integer-valued, we obtain by Markov's inequality
    \begin{equation}
        \label{eq:FirstMomentMethod1}
        \P{S_n(\varepsilon) = 0} = 1 - \P{S_n(\varepsilon) \geq 1} \geq 1 - \E{S_n(\varepsilon)}.
    \end{equation}
    Set $\varepsilon = \beta2e^{-3/4}\sqrt{nm}4^{-\xi n/m^2}$. Then, from \cref{lem:SmallNormProbability} and \cref{lem:FirstMoment} follows that
    \begin{equation}
        \label{eq:FirstMomentMethod2}
        \E{S_n(\varepsilon)} = 2^n \P{\norm{X} \leq \frac{\varepsilon}{\sqrt{n}}} = \beta^{\xi^{-1}m^2/2}.
    \end{equation}
    Combining \eqref{eq:FirstMomentMethod1} and \eqref{eq:FirstMomentMethod2} yields the lower bound in Part \ref{thm:GaussianDiscrepancy1} of \cref{thm:GaussianDiscrepancy}. For the lower bound in Part \ref{thm:GaussianDiscrepancy2} note that $\beta^{\xi^{-1}m^2/2} = o(1)$ in the setting $m = \omega(1)$.

    \subsection{Upper bound via the second moment method}
    \label{sec:SecondMomentBound}

    The second moment method leverages the relationship between the first and second moments of a random variable to lower bound its probability of beeing positive. Despite its simple nature, it is a powerful tool in combinatorics; for example, many applications are presented in the book by Alon and Spencer \cite{Alon16}. It essentially consists in an application of the Paley-Zygmund inequality \cite{Paley32}. This yields
    \begin{equation}
        \label{eq:SecondMomentMethod}
        \P{S_n(\varepsilon) > 0} \geq \frac{\E{S_n(\varepsilon)}^2}{\E{S_n(\varepsilon)^2}}
    \end{equation}
    and shows that a uniform bound of the form $\P{S_n(\varepsilon) > 0} \geq 1/c$ holds when $\E{S_n(\varepsilon)^2}/\E{S_n(\varepsilon)}^2 \allowbreak \leq c$ for some $c > 0$. In particular, $\E{S_n(\varepsilon)^2}/\E{S_n(\varepsilon)}^2 \lesssim 1$ implies that $\P{S_n(\varepsilon) > 0} \approx 1$. The following lemma gives a useful representation for the second moment of $S_n(\varepsilon)$.

    \begin{lemma}
        \label{lem:SecondMoment}
        The second moment of $S_n(\varepsilon)$ is given by
        \begin{equation*}
            \E S_n(\varepsilon)^2 = 2^n \sum_{k=0}^n \binom{n}{k} \P{\norm{X_k} \leq \frac{\varepsilon}{\sqrt{n}},\norm{Y_k} \leq \frac{\varepsilon}{\sqrt{n}}},
        \end{equation*}
        where $X_k,Y_k$ are two $m \times m$ GOE matrices having correlation $\rho_k = 1 - 2k/n$.
    \end{lemma}

    \begin{proof}
        By linearity of expectation, we have that
        \begin{equation*}
            \E S_n(\varepsilon)^2 = \sum_{x \in \{\pm1\}^n} \sum_{y \in \{\pm1\}^n} \P{\norm{\frac{1}{\sqrt{n}} \sum_{i=1}^n x_iA_i} \leq \frac{\varepsilon}{\sqrt{n}},\norm{\frac{1}{\sqrt{n}} \sum_{i=1}^n y_iA_i} \leq \frac{\varepsilon}{\sqrt{n}}}.
        \end{equation*}
        Consider two signings $x,y \in \{\pm1\}^n$ with Hamming distance $k$, that is, $x$ and $y$ differ in exactly $k$ entries. Then $X_k = \frac{1}{\sqrt{n}} \sum_{i=1}^n x_iA_i$ and $Y_k = \frac{1}{\sqrt{n}} \sum_{i=1}^n y_iA_i$ are two GOE matrices with correlation
        \begin{equation*}
            \corr(X_{ij},Y_{ij}) = \frac{1}{n} \sum_{s=1}^n \sum_{t=1}^n x_sy_t \E{(A_s)_{ij}(A_t)_{ij}} = \frac{1}{n} \sum_{s=1}^n x_sy_s = 1 - \frac{2k}{n} = \rho_k.
        \end{equation*}
        Fix a signing $x \in \{\pm1\}^n$. Then for $k = 0,\ldots,n$ there are exactly $\binom{n}{k}$ many signings $y \in \{\pm1\}^n$ having a Hamming distance of $k$ to $x$. It follows that
        \begin{equation*}
            \E S_n(\varepsilon)^2 = 2^n \sum_{k=0}^n \binom{n}{k} \P{\norm{X_k} \leq \frac{\varepsilon}{\sqrt{n}},\norm{Y_k} \leq \frac{\varepsilon}{\sqrt{n}}}.
            \qedhere
        \end{equation*}
    \end{proof}

    In order to describe the relationship between $\E{S_n(\varepsilon)}^2$ and $\E{S_n(\varepsilon)^2}$, we define the probability ratio
    \begin{equation}
        \label{eq:ProbabilityRatio}
        R_k(\delta) \colonequals \frac{\P{\norm{X_k} \leq \delta,\norm{Y_k} \leq \delta}}{\P{\norm{X} \leq \delta}^2}
    \end{equation}
    for $k = 0,\ldots,n$, where $X$ is an $m \times m$ GOE matrix and $X_k,Y_k$ are two $m \times m$ GOE matrices with correlation coefficent
    \begin{equation}
        \label{eq:CorrelationCoefficient}
        \rho_k \colonequals 1 - \frac{2k}{n}.
    \end{equation}
    In view of \cref{lem:FirstMoment} and \cref{lem:SecondMoment}, for $\varepsilon = \sqrt{n}\delta$, we have that
    \begin{equation}
        \label{eq:MomentRelationship}
        \E{S_n(\varepsilon)^2} = \E{S_n(\varepsilon)}^2 2^{-n} \sum_{k=0}^n \binom{n}{k} R_k(\delta),
    \end{equation}
    and bounding $\E{S_n(\varepsilon)}^2/\E{S_n(\varepsilon)^2}$ reduces to bounding the ratio
    \begin{equation}
        \label{eq:RemainderTerm}
        2^{-n} \sum_{k=0}^n \binom{n}{k} R_k(\delta).
    \end{equation}
    Note that if $R_k(\delta) = 1$ for all $k$, this would equal exactly one; our goal will be to show that the $R_k(\delta)$ with $k$ close to $n/2$, which are weighted most heavily by the binomial coefficient, are not too large and that this is actually approximately true.

    For the analysis, we follow the approach of Turner, Meka and Rigollet \cite{Turner20} and apply a truncation argument to split the sum into a leading-order term and a lower-order term. A careful analysis of the two terms, which is postponed to the next subsections, leads to the following results.
    The first lemma shows that the contribution from the lower-order term, which consists of the summands with $k \leq n/4$ or $k \geq 3n/4$, is negligible.

    \begin{lemma}
        \label{lem:LowerOrderTerm}
        Let $\delta = \gamma2e^{-3/4}\sqrt{m}4^{-\xi n/m^2}$ for some constant $\gamma > 1$.
        \begin{enumerate}[label=(\alph*)]
            \item In the setting $\omega(1) = m^2 = o(n)$, we have
            \begin{equation*}
                2^{-n} \sum_{k=0}^{n/4} \binom{n}{k} R_k(\delta) + 2^{-n} \sum_{k=3n/4}^n \binom{n}{k} R_k(\delta) = o(1).
            \end{equation*}
            \label{lem:LowerOrderTerm1}
            \item In the setting $m = O(1)$, we have
            \begin{equation*}
                2^{-n} \sum_{k=0}^{n/4} \binom{n}{k} R_k(\delta) + 2^{-n} \sum_{k=3n/4}^n \binom{n}{k} R_k(\delta) \leq  2\gamma^{-\xi^{-1} m^2/2} + o(1).
            \end{equation*}
            \label{lem:LowerOrderTerm2}
        \end{enumerate}
    \end{lemma}

    The second lemma shows that the main contribution comes from the leading-order term, consisting of the summands with $n/4 \leq k \leq 3n/4$.

    \begin{lemma}
        \label{lem:LeadingOrderTerm}
        Let $\delta = \gamma2e^{-3/4}\sqrt{m}4^{-\xi n/m^2}$ for some constant $\gamma > 1$. In the setting $m^2 \ll n/\log{n}$, we have
        \begin{equation*}
            2^{-n} \sum_{k=n/4}^{3n/4} \binom{n}{k} R_k(\delta) = 1 + o(1).
        \end{equation*}
    \end{lemma}

    To avoid confusion, in these results we interpret $\sum_{k=a}^b$ as the sum over all integers from $\lceil a \rceil$ to $\lfloor b \rfloor$.
    This may lead to an overlap between the sums, but that does not invalidate our bounds on the full sum since all summands are nonnegative.
    The proofs of \cref{lem:LowerOrderTerm} and \cref{lem:LeadingOrderTerm} are subject of the remainder of this section, but let us first show how they finish the proof of \cref{thm:GaussianDiscrepancy}.

    \begin{proof}[Proof of \cref{thm:GaussianDiscrepancy}]
        The lower bounds were already derived in \cref{sec:FirstMomentBound}, and it remains to show that the claimed upper bounds hold. For Part \ref{thm:GaussianDiscrepancy1} we assume that $\omega(1) = m^2 \ll n/\log{n}$. Set $\delta = \gamma2e^{-3/4}\sqrt{m}4^{-\xi n/m^2}$ and $\varepsilon = \sqrt{n}\delta$. Combining \cref{lem:LowerOrderTerm} and \cref{lem:LeadingOrderTerm} shows that the remainder term \eqref{eq:RemainderTerm} is bounded by $1 + o(1)$. From \eqref{eq:MomentRelationship} follows that $\E{S_n(\varepsilon)^2} \lesssim \E{S_n(\varepsilon)}^2$ and an application of the Paley-Zygmund inequality \eqref{eq:SecondMomentMethod} yields that $S_n(\varepsilon) > 0$ with high probability. Recalling that $S_n(\varepsilon) > 0$ corresponds to the event $\disc(A_1,\ldots,A_n) \leq \varepsilon$ completes the proof of Part \ref{thm:GaussianDiscrepancy1}. The proof of Part \ref{thm:GaussianDiscrepancy2} follows along the same lines.
    \end{proof}

    \subsection{Bounding the lower-order term}

    The goal of this subsection is to prove the statement of \cref{lem:LowerOrderTerm}, namely that, if we set $\delta \colonequals \gamma2e^{-3/4}\sqrt{m}4^{-\xi n/m^2}$ with $\gamma > 1$, then the lower-order term
    \begin{equation}
        \label{eq:LowerOrderTerm}
        2^{-n} \sum_{k=0}^{n/4} \binom{n}{k} R_k(\delta) + 2^{-n} \sum_{k=3n/4}^n \binom{n}{k} R_k(\delta)
    \end{equation}
    is asymptotically vanishing. This applies to the increasing dimension setting in Part \ref{lem:LowerOrderTerm1}. The asymptotic bound in Part \ref{lem:LowerOrderTerm2} for the constant dimension setting follows in a similar manner. The key idea for bounding \eqref{eq:LowerOrderTerm} is to give a uniform estimate on $R_k(\delta)$ to extract the probability ratios, and then use the following inequality to bound the remaining sum of binomial coefficients.

    \begin{lemma}
        \label{lem:BinomialSum}
        Let $t,n \in \N$ with $t \leq n$. Then
        \begin{equation*}
            \sum_{k=0}^t \binom{n}{k} = \sum_{k=n-t}^n \binom{n}{k} \leq \left(\frac{en}{t}\right)^t.
        \end{equation*}
    \end{lemma}

    \begin{proof}
        The first equation follows by symmetry. Applying the binomial theorem shows that
        \begin{equation*}
            \sum_{k=0}^t \binom{n}{k} \left(\frac{t}{n}\right)^t \leq \sum_{k=0}^n \binom{n}{k} \left(\frac{t}{n}\right)^k = \left(1 + \frac{t}{n}\right)^n \leq e^t.
        \end{equation*}
        Multiplying both sides by $(n/t)^t$ yields the claim.
    \end{proof}

    Bounding the probability ratios $R_k(\delta)$ is somewhat delicate, as different estimates are required depending on the position of $k$. If $k$ is close to $0$ or $n$ (and thus $\binom{n}{k}$ is relatively small), the crude bound
    \begin{equation*}
        \P{\norm{X_k} \leq \delta,\norm{Y_k} \leq \delta} \leq \P{\norm{X} \leq \delta}
    \end{equation*}
    in connection with \cref{lem:SmallNormProbability} yields a sufficiently accurate estimate
    \begin{equation}
        \label{eq:ProbabilityRatioBound}
        R_k(\delta) = \frac{\P{\norm{X_k} \leq \delta,\norm{Y_k} \leq \delta}}{\P{\norm{X} \leq \delta}^2} \leq \P{\norm{X} \leq \delta}^{-1} = 2^n \gamma^{-\xi^{-1} m^2/2}.
    \end{equation}
    As $k$ tends towards $n/2$, the contribution of the term $R_k(\delta)$ to \eqref{eq:LowerOrderTerm} becomes larger, due to the growth of the binomial coefficent, and a stronger estimate on $R_k(\delta)$ is necessary. Using the results on the probability density functions in \cref{lem:GaussianEigenvalues1} and \cref{lem:GaussianEigenvalues2}, we obtain such a bound in the next lemma.

    \begin{lemma}
        \label{lem:ProbabilityRatio}
        Let $R_k(\delta)$ be defined as in \eqref{eq:ProbabilityRatio}. For $0 < k < n$ holds
        \begin{equation*}
            R_k(\delta) \leq (1 - \rho_k^2)^{-m(m+1)/4} \exp{\abs{\rho_k}m\delta^2}.
        \end{equation*}
    \end{lemma}

    \begin{proof}
        Recall the statements of \cref{lem:GaussianEigenvalues1} and \cref{lem:GaussianEigenvalues2}. By \cref{lem:GaussianEigenvalues1}, the ordered eigenvalues of an $m \times m$ GOE matrix $X$ have joint density
        \begin{equation*}
            p(\lambda) \colonequals C_m \exp{-\frac{\norm{\lambda}_2^2}{4}} \Delta(\lambda),
        \end{equation*}
        and by \cref{lem:GaussianEigenvalues2}, the ordered eigenvalues two $m \times m$ GOE matrices $X_k,Y_k$ with correlation $\rho_k$ have joint density bounded by
        \begin{equation*}
            q_k(\lambda,\mu) \colonequals C_m^2 (1 - \rho_k^2)^{-m(m + 1)/4} \exp{-\frac{\norm{\lambda}_2^2 - 2\rho_k\inp{\lambda}{\mu} + \norm{\mu}_2^2}{4(1 - \rho_k^2)}} \Delta(\lambda)\Delta(\mu).
        \end{equation*}
        To express the relationship between $p$ and $q_k$, we define the remainder function
        \begin{equation*}
            r_k(\lambda,\mu) \colonequals (1 - \rho_k^2)^{-m(m + 1)/4} \exp{-\frac{\rho_k^2\norm{\lambda}_2^2 - 2\rho_k\inp{\lambda}{\mu} + \rho_k^2\norm{\mu}_2^2}{4(1 - \rho_k^2)}}
        \end{equation*}
        and observe that $q_k(\lambda,\mu) = p(\lambda)p(\mu)r_k(\lambda,\mu)$. By the same argument as in the proof of \cref{lem:SmallNormProbability}, namely that the spectral norm of a symmetric matrix corresponds to the maximal absolute value of its eigenvalues, it follows that
        \begin{equation*}
            \P{\norm{X} \leq \delta} = \int_{D_\delta} p(\lambda) d\lambda, \quad \P{\norm{X_k} \leq \delta,\norm{Y_k} \leq \delta} \leq \int_{D_\delta \times D_\delta} q_k(\lambda,\mu) d(\lambda,\mu),
        \end{equation*}
        where $D_\delta \colonequals \{\lambda \in \R^m : -\delta \leq \lambda_1 \leq \ldots \leq \lambda_m \leq \delta\}$. This yields the simple estimate
        \begin{equation}
            \label{eq:ProbabilityRatioProof1}
            R_k(\delta) \leq \max_{\lambda,\mu \in D_\delta} r_k(\lambda,\mu).
        \end{equation}
        The remaining step is to determine the maximum on the right-hand side. Consider the case $k \leq n/2$. Then $\rho_k \geq 0$ and maximizing $r_k(\lambda,\mu)$ amounts to maximizing
        \begin{equation*}
            f(\lambda,\mu) \colonequals 2\inp{\lambda}{\mu} - \rho_k(\norm{\lambda}_2^2 + \norm{\mu}_2^2).
        \end{equation*}
        Using standard optimization techniques, one can argue\footnote{For example, one may argue as follows. It is not hard to see that $r_k$ is concave as $\rho_k < 1$ due to the assumption $k > 0$. Since $D_\delta$ is a polytope and in particular convex, it follows that the KKT conditions provide necessary and sufficient conditions for optimality \cite{Boyd04}. Verifying that $(\delta,\ldots,\delta,\delta,\ldots,\delta)^T$ satisfies the KKT conditions yields the claim.} that $f(\lambda,\mu)$ attains its maximum over $D_\delta \times D_\delta$ at $(\lambda^\star,\mu^\star) = (\delta,\ldots,\delta,\delta,\ldots,\delta)^T$. In particular, we have
        \begin{equation}
            \label{eq:ProbabilityRatioProof2}
            \max_{\lambda,\mu \in D_\delta} r_k(\lambda,\mu) = r_k(\lambda^\star,\mu^\star) = (1 - \rho_k^2)^{-m(m + 1)/4} \exp{\frac{\rho_k m\delta^2}{2(1 + \rho_k)}}.
        \end{equation}
        Similarly for the case $k \geq n/2$, one can show that $r_k(\lambda,\mu)$ attains its maximum over $D_\delta \times D_\delta$ at $(\lambda^\star,\mu^\star) = (\delta,\ldots,\delta,-\delta,\ldots,-\delta)^T$ and hence
        \begin{equation}
            \label{eq:ProbabilityRatioProof3}
            \max_{\lambda,\mu \in D_\delta} r_k(\lambda,\mu) = r_k(\lambda^\star,\mu^\star) = (1 - \rho_k^2)^{-m(m+1)/4} \exp{-\frac{\rho_k m\delta^2}{2(1 - \rho_k)}}.
        \end{equation}
        Combining our observations in \eqref{eq:ProbabilityRatioProof2} and \eqref{eq:ProbabilityRatioProof3} with \eqref{eq:ProbabilityRatioProof1} completes the proof.
    \end{proof}

    Now we are ready to prove \cref{lem:LowerOrderTerm}.

    \begin{proof}[Proof of \cref{lem:LowerOrderTerm}]
        For Part \ref{lem:LowerOrderTerm1} we assume that $\omega(1) = m^2 = o(n)$. Due to symmetry $R_k(\delta) = R_{n-k}(\delta)$, it suffices to show that the sum over $0 \leq k \leq n/4$ is of order $o(1)$. Denote $\alpha \colonequals n/m^2$ and notice that $\omega(1) = \alpha = o(n)$. For the further analysis, we split the sum into two parts
        \begin{equation*}
            S_1 \colonequals 2^{-n} \sum_{k=0}^{n/\alpha^2} \binom{n}{k} R_k(\delta), \quad S_2 \colonequals 2^{-n} \sum_{k=n/\alpha^2}^{n/4} \binom{n}{k} R_k(\delta).
        \end{equation*}
        First, consider the sum $S_1$. Applying \cref{lem:BinomialSum} with $t = n/\alpha^2$ yields the estimate
        \begin{equation}
            \label{eq:LowerTermProof2}
            \sum_{k=0}^{n/\alpha^2} \binom{n}{k} \leq (\alpha^2e)^{n/\alpha^2}.
        \end{equation}
        Combining \eqref{eq:LowerTermProof2} with the bound in \eqref{eq:ProbabilityRatioBound} shows that
        \begin{equation*}
            S_1 \leq \exp{\frac{(2\log{\alpha} + 1)n}{\alpha^2} - \frac{\log(\gamma)\xi^{-1}m^2}{2}} = o(1),
        \end{equation*}
        where we used $m^2 = n/\alpha$ and $\log{\gamma} > 0$ to conclude the bound. Next, we consider the sum $S_2$. For $n/\alpha^2 \leq k \leq n/4$, we have that $\abs{\rho_k} \leq 1$ and $1 - \rho_k^2 \geq \alpha^{-2}$, and thus by \cref{lem:ProbabilityRatio}
        \begin{equation}
            \label{eq:LowerTermProof3}
            R_k(\delta) \leq (1 - \rho_k^2)^{-m(m + 1)/4} \exp{\abs{\rho_k}m\delta^2} \leq \alpha^{m(m + 1)/2} \exp{m\delta^2}.
        \end{equation}
        Applying \cref{lem:BinomialSum} with $t = n/4$ and noting that $2^{-n}(4e)^{n/4} \leq 2^{-n/8}$ yields the estimate
        \begin{equation}
            \label{eq:LowerTermProof4}
            2^{-n}\sum_{k=n/\alpha^2}^{n/4} \binom{n}{k} \leq 2^{-n/8}.
        \end{equation}
        Combining the bounds in \eqref{eq:LowerTermProof3} and \eqref{eq:LowerTermProof4} shows that
        \begin{equation*}
            S_2 \leq \exp{\frac{\log(\alpha)m(m + 1)}{2} + m\delta^2 - \frac{\log(2)n}{8}} = o(1),
        \end{equation*}
        where we used $m\delta^2 = m^2o(1)$ and $m^2 = n/\alpha$ to conclude the bound. Altogether, we proved that $S_1 + S_2 = o(1)$ and the claim follows. The proof of Part \ref{lem:LowerOrderTerm2} follows along the same lines.
    \end{proof}

    \subsection{Bounding the leading-order term}

    The goal of this subsection is to prove the statement of \cref{lem:LeadingOrderTerm}, namely that, if we set $\delta \colonequals \gamma2e^{-3/4}\sqrt{m}4^{-\xi n/m^2}$ with $\gamma > 1$, then the leading-order term
    \begin{equation}
        \label{eq:LeadingOrderTerm}
        2^{-n} \sum_{k=n/4}^{3n/4} \binom{n}{k} R_k(\delta)
    \end{equation}
    is asymptotically $1 + o(1)$, provided that $m^2 \ll n/\log{n}$. Following the approach of Turner, Meka and Rigollet \cite{Turner20}, we approximate the above sum by an integral of the form $\int_a^b \exp{n\varphi(x)} dx$ and then use the Laplace method to obtain a sharp bound.
    We begin with a brief description of this asymptotic method.

    \textbf{Laplace method.} Let $\varphi : [a,b] \to \R$ be a twice continuously differentiable function with a unique maximum at $y \in (a,b)$. Furthermore, assume that the second derivative of $\varphi$ is negative at $y$. Consider the integral
    \begin{equation}
        \label{eq:LaplaceMethod1}
        \int_a^b \exp{n\varphi(x)} dx.
    \end{equation}
    To determine the asymptotic behaviour of \eqref{eq:LaplaceMethod1}, observe that the main contribution comes from a small neighborhood around $y$, that is,
    \begin{equation}
        \label{eq:LaplaceMethod2}
        \int_a^b \exp{n\varphi(x)} dx \approx \int_{y - \delta}^{y + \delta} \exp{n\varphi(x)} dx
    \end{equation}
    for some $\delta = o(1)$. A second order Taylor expansion around $y$ gives
    \begin{equation}
        \label{eq:LaplaceMethod3}
        \varphi(x) \approx \varphi(y) + \varphi'(y)(x - y) + \frac{1}{2} \varphi''(y)(x - y)^2
    \end{equation}
    for $x \in [y - \delta,y + \delta]$. Since $y$ is assumed to be an interior point of $[a,b]$ at which $\varphi$ admits a maximum, it follows that $y$ is a stationary point, that is, $\varphi'(y) = 0$. Furthermore, by assumption $\varphi''(y) < 0$. Therefore, we can rewrite \eqref{eq:LaplaceMethod3} as
    \begin{equation*}
        \varphi(x) \approx \varphi(y) - \frac{1}{2} \abs{\varphi''(y)}(x - y)^2.
    \end{equation*}
    Substituting this into \eqref{eq:LaplaceMethod2} yields the integral approximation
    \begin{equation*}
        \int_a^b \exp{n\varphi(x)} dx \approx \exp{n\varphi(y)} \int_{y - \delta}^{y + \delta} \exp{-\frac{1}{2}n\abs{\varphi''(y)}(x - y)^2} dx.
    \end{equation*}
    Due to the exponential decay, we can replace the integral boundaries by $-\infty$ and $\infty$ to obtain a Gaussian integral. This can be evaluated in closed form and yields
    \begin{equation}
        \label{eq:LaplaceMethod4}
        \int_a^b \exp{n\varphi(x)} dx \approx \sqrt{\frac{2\pi}{n\abs{\varphi''(y)}}} \exp{n\varphi(y)}.
    \end{equation}
    For further details on the Laplace method, we refer to the book by de Bruijn \cite{Bruijn81}. Under some additional assumptions, we can apply the method to sequences of functions $\varphi_n : [a,b] \to \R$.
    Apart from being twice continuously differentiable and having a unique maximum at $y \in (a,b)$ with $\varphi_n''(y) < 0$, we require that $\varphi_n''$ is equicontinuous at $y$, that is, that for each $\varepsilon > 0$ there exists $\delta > 0$ such that
    \begin{equation}
        \label{eq:Equicontinuity}
        \abs{\varphi_n''(x) - \varphi_n''(y)} \leq \varepsilon
    \end{equation}
    for all $x \in [y - \delta,y + \delta]$ and $n \in \N$. This allows us to perform a second order Taylor expansion around $y$ with an error independent of $n$. Furthermore, we require that for a sufficiently small $\delta > 0$ the difference between $\varphi_n(y)$ and the maximum of $\varphi_n(x)$ over $[a,y - \delta] \cup [y + \delta,b]$ is bounded by some constant independent of $n$. This ensures that the integral of $\exp{n\varphi_n(x)}$ is asymptotically negligible compared to $\exp{n\varphi_n(y)}/\sqrt{n}$. For concave functions the maximum will be attained at one of the boundary points $y - \delta$ and $y + \delta$. In connection with the equicontinuity at $y$, it suffices to require that $\varphi_n''(y)$ is bounded by some negative constant independent of $n$. We obtain the following result, whose proof is deferred to \cref{sec:Appendix}.

    \begin{lemma}
        \label{lem:LaplaceMethod}
        Let $\varphi_n : [a,b] \to \R$ be a sequence of twice continuously differentiable and concave functions with a unique maximum at $y \in (a,b)$. Assume that $\varphi_n''$ is equicontinuous at $y$ and $\varphi_n''(y) \leq c$ for some $c < 0$ independent of $n$. Then
        \begin{equation}
            \label{eq:LaplaceMethod}
            \int_a^b \exp{n\varphi_n(x)} dx \approx \sqrt{\frac{2\pi}{n\abs{\varphi_n''(y)}}} \exp{n\varphi_n(y)}.
        \end{equation}
    \end{lemma}

    To apply Laplace's method in our situation, we need to bound the leading-order term \eqref{eq:LeadingOrderTerm} by an integral of the form \eqref{eq:LaplaceMethod1}. By \cref{lem:ProbabilityRatio}, we have
    \begin{equation*}
        R_k(\delta) \leq (1 - \rho_k^2)^{-m(m+1)/4} \exp{\abs{\rho_k}m\delta^2}.
    \end{equation*}
    We replace $\abs{\rho_k}$ by $\sqrt{\rho_k^2 + \varepsilon}$ with $\varepsilon \colonequals m^2\delta^4 > 0$ to obtain a twice continuously differentiable approximation (for the approximation argument $\varepsilon > 0$ suffices, but as it turns out $\sqrt{\varepsilon} = \omega(m\delta^2/n)$ is necessary for \cref{lem:StrictConcavity} to hold)
    \begin{equation}
        \label{eq:LaplaceApproximation1}
        R_k(\delta) \leq (1 - \rho_k^2)^{-m(m+1)/4} \exp{(\rho_k^2 + \varepsilon)^{1/2}m\delta^2}.
    \end{equation}
    Recall from the definition of $\rho_k$ in \eqref{eq:CorrelationCoefficient} that
    \begin{equation*}
        1 - \rho_k^2 = \frac{4k}{n}\left(1 - \frac{k}{n}\right).
    \end{equation*}
    Introducing the auxiliary functions $f,g : [1/4,3/4] \to \R$, defined by
    \begin{equation*}
        f(x) \colonequals ((1 - 2x)^2 + \varepsilon)^{1/2}, \quad g(x) \colonequals -\log(x(1 - x)),
    \end{equation*}
    and denoting $x_k \colonequals k/n$ for $k = 0,\ldots,n$, we can rewrite \eqref{eq:LaplaceApproximation1} as
    \begin{equation}
        \label{eq:LaplaceApproximation2}
        R_k(\delta) \leq \exp{\frac{m(m + 1)}{4} (g(x_k) - \log{4}) + m\delta^2 f(x_k)}.
    \end{equation}
    To approximate the binomial coefficients, one can use the standard bound
    \begin{equation*}
        \binom{n}{k} \leq \exp{nh(k/n)},
    \end{equation*}
    where $h : [0,1] \to \R$ is the binary entropy function defined by $h(0) = h(1) = 0$ and
    \begin{equation*}
        h(x) = -x\log{x} - (1 - x)\log(1 - x), \quad x \in (0,1).
    \end{equation*}
    However, proceeding with this approximation leads to an insufficient bound of $(1 + o(1))\sqrt{2\pi n}$. To get rid of the factor $\sqrt{2\pi n}$, the following sharper approximation is required, which was also used by Turner, Meka and Rigollet \cite{Turner20}.

    \begin{lemma}
        Let $n,k \in \N$ with $k \leq n$. For $x = k/n$ holds
        \begin{equation}
            \label{eq:BinomialApproximation1}
            \binom{n}{k} \approx \frac{\exp{nh(x)}}{\sqrt{2\pi nx(1-x)}}.
        \end{equation}
    \end{lemma}

    \begin{proof}
        Applying Stirling's approximation $n! \approx \sqrt{2\pi n} (n/e)^n$ gives
        \begin{equation*}
            \binom{n}{k} = \frac{n!}{k!(n-k)!} \approx \sqrt{\frac{n}{2\pi k(n-k)}} \left(\frac{n}{k}\right)^k \left(\frac{n}{n-k}\right)^{n-k}.
        \end{equation*}
        This expression can be rewritten as
        \begin{equation*}
            \frac{1}{\sqrt{2\pi k(1-k/n)}} \exp{-\log(k/n)k-\log(1-k/n)(n-k)}
        \end{equation*}
        and substituting $x$ for $n/k$ yields the desired bound.
    \end{proof}

    Using the previously introduced functions, we can rewrite \eqref{eq:BinomialApproximation1} as
    \begin{equation}
        \label{eq:BinomialApproximation2}
        \binom{n}{k} \approx \frac{1}{\sqrt{2\pi n}} \exp{nh(x_k) + \frac{1}{2} g(x_k)}.
    \end{equation}
    Now define the twice continuously differentiable function $\varphi_n : [1/4,3/4] \to \R$ by
    \begin{equation}
        \label{eq:LaplaceFunction}
        \varphi_n(x) \colonequals h(x) + \frac{m(m+1)}{4n} (g(x) - \log{4}) + \frac{m\delta^2}{n} f(x) + \frac{1}{2n} g(x).
    \end{equation}
    In view of \eqref{eq:LaplaceApproximation2} and \eqref{eq:BinomialApproximation2}, we conclude $\binom{n}{k} R_k(\delta) \lesssim \frac{1}{\sqrt{2\pi n}} \exp{n\varphi_n(x_k)}$. Summing up this inequality over $k = n/4,\ldots,3n/4$, we obtain
    \begin{equation}
        \label{eq:LaplaceApproximation3}
        \sum_{k=n/4}^{3n/4} \binom{n}{k} R_k(\delta) \lesssim \frac{1}{\sqrt{2\pi n}} \sum_{k=n/4}^{3n/4} \exp{n\varphi_n(x_k)}.
    \end{equation}
    In accordance with the previous convention, we replace $n/4,3n/4$ by $\lceil n/4 \rceil,\lfloor 3n/4 \rfloor$ if the values are not integer. The next step is to show that $\varphi_n$ is strictly concave for sufficiently large $n$. This yields the desired integral approximation and allows us to apply Laplace's method.

    \begin{lemma}
        \label{lem:StrictConcavity}
        In the setting of \cref{lem:LeadingOrderTerm}, the function $\varphi_n$ defined in \eqref{eq:LaplaceFunction} is strictly concave and has a unique maximum at $1/2$ for $n$ large enough. Furthermore, its second derivative $\varphi_n''$ converges uniformly to $h''$ over $[1/4,3/4]$.
    \end{lemma}

    \begin{proof}
        Throughout the proof, assume that $h$ is restricted to the interval $[1/4,3/4]$. The second derivatives of $f,g,h$ are given by
        \begin{equation*}
            f''(x) = \frac{4\varepsilon}{((1 - 2x)^2 + \varepsilon)^{3/2}}, \quad g''(x) = \frac{2x^2 - 2x + 1}{(x - 1)^2x^2}, \quad h''(x) = -\frac{1}{x(1 - x)}
        \end{equation*}
        for $x \in [1/4,3/4]$. By linearity,
        \begin{equation*}
            \varphi_n''(x) = h''(x) + \frac{m(m+1)}{4n} g''(x)+ \frac{m\delta^2}{n} f''(x) + \frac{1}{2n} g''(x)
        \end{equation*}
        for $x \in [1/4,3/4]$. Note that $g''$ is continuous and therefore bounded on $[1/4,3/4]$, and $f''$ is non-negative and bounded by $4/\sqrt{\varepsilon} = 4/(m\delta^2)$.
        Using the assumption $m^2 = o(n)$, we conclude that $\varphi_n''$ converges pointwise to $h''$. Since $f'',g'',h''$ are continuously differentiable, it follows that $\varphi_n''$ is $L$-Lipschitz for some $L$ not depending on $n$. So the above convergence is uniform, which proves the second part of the theorem.

        For the first part, observe that $h''(x) \leq -4$ for all $x \in [1/4,3/4]$ and hence $\varphi_n''$ is negative for $n$ large enough. This implies that $\varphi_n$ is strictly concave and therefore admits a unique maximum. Verifying that $\varphi_n'(1/2) = 0$, we find that this maximum occurs at $1/2$.
    \end{proof}

    \begin{proof}[Proof of \cref{lem:LeadingOrderTerm}]
        According to the discussion above, resulting in \eqref{eq:LaplaceApproximation3}, we have
        \begin{equation}
            \label{eq:LeadingTermProof1}
            2^{-n} \sum_{k=n/4}^{3n/4} \binom{n}{k} R_k(\delta) \lesssim \frac{2^{-n}}{\sqrt{2\pi n}} \sum_{k=n/4}^{3n/4} \exp{n\varphi_n(x_k)},
        \end{equation}
        where $\varphi_n$ is defined as in \eqref{eq:LaplaceFunction}. By \cref{lem:StrictConcavity}, for $n$ sufficiently large the function $\varphi_n$ is strictly concave with a unique maximum at $1/2$. Further, $\varphi_n$ is symmetric around $1/2$.
        Therefore, a Riemann sum of $\exp{\varphi_n(x/n)}$ over $[n/4,3n/4]$ with partition $x_{n/4},\ldots,x_{3n/4}$ underestimates the corresponding integral
        \begin{equation}
            \label{eq:LeadingTermProof2}
            \sum_{k=n/4}^{3n/4} \exp{n\varphi_n(x_k)} \leq n \int_{1/4}^{3/4} \exp{n\varphi_n(x)} dx.
        \end{equation}
        By Lemma \ref{lem:StrictConcavity}, we also know that $\varphi_n''$ converges uniformly to $h''$ over $[1/4,3/4]$. This implies that $\varphi_n''$ is equicontinuous and bounded by some negative constant if $n$ is large enough. Thus, we can apply \cref{lem:LaplaceMethod} to obtain
        \begin{equation}
            \label{eq:LeadingTermProof3}
            \int_{1/4}^{3/4} \exp{n\varphi_n(x)} dx \approx \sqrt{\frac{2\pi}{n\abs{\varphi_n''(1/2)}}} \exp{n\varphi_n(1/2)}.
        \end{equation}
        The convergence of $\varphi_n''(1/2)$ to $h''(1/2)$ yields $\varphi_n''(1/2) \approx -4$ and a straightforward calculation gives $\varphi_n(1/2) = \log{2} + \sqrt{\varepsilon}m\delta^2/n + \log{2}/n$. Combining \eqref{eq:LeadingTermProof1}, \eqref{eq:LeadingTermProof2}, \eqref{eq:LeadingTermProof3} and substituting $\varphi_n(1/2)$, $\varphi_n''(1/2)$ shows that
        \begin{equation*}
            2^{-n} \sum_{k=n/4}^{3n/4} \binom{n}{k} R_k(\delta) \lesssim \frac{2^{-n}}{\sqrt{\abs{\varphi_n''(1/2)}}} \exp{n\varphi_n(1/2)} \approx \exp{\sqrt{\varepsilon}m\delta^2}.
        \end{equation*}
        Since $\delta^4 = O(m^2/n^2)$ if $m^2 \ll n/\log{n}$ (in fact $m^2 \leq 2\log(4)\xi n/\log{n}$ suffices), we have that $\sqrt{\varepsilon}m\delta^2 = O(m^4/n^2) = o(1)$ and the desired result follows.
    \end{proof}

    \section{Discrepancy of general random matrices}
    \label{sec:GeneralDiscrepancy}

    In this section, we prove \cref{thm:GeneralDiscrepancy} and \cref{thm:WishartDiscrepancy}. We first introduce our proof techniques for lower and upper bounds on matrix discrepancy.

    \subsection{Lower bound via the Gramian spectral method}

    Our approach to a lower bound is based on a connection between the discrepancy of $A_1,\ldots,A_n$ and the spectrum of its Gram matrix with respect to the Frobenius inner product. Given vectors $v_1,\ldots,v_n$ in an inner product space, we refer to the $n \times n$ matrix with entries $\inp{v_i}{v_j}$ for $1 \leq i,j \leq n$ as the \emph{Gram matrix} of $v_1,\ldots,v_n$.

    \begin{lemma}
        \label{lem:GramBound}
        Let $A_1,\ldots,A_n \in \R^{m \times m}$ be symmetric matrices, and let $M$ denote the Gram matrix of $A_1,\ldots,A_n$ with respect to the Frobenius inner product. Then
        \begin{equation*}
            \sqrt{\frac{n}{m}\lambda_{\min}(M)} \leq \disc(A_1,\ldots,A_n) \leq \sqrt{n\lambda_{\max}(M)}.
        \end{equation*}
    \end{lemma}

    \begin{proof}
        Recall that $\norm{A} \leq \norm{A}_F \leq \sqrt{m}\norm{A}$. For $x \in \R^n$ follows that
        \begin{equation*}
            \norm{\sum_{i=1}^n x_iA_i}^2 \geq \frac{1}{m} \norm{\sum_{i=1}^n x_iA_i}_F^2 = \frac{1}{m} \sum_{i,j=1}^n x_ix_j\inp{A_i}{A_j} = \frac{1}{m} \inp{x}{Mx}
        \end{equation*}
        and similarly we derive $\norm{\sum_{i=1}^n x_iA_i}^2 \leq \inp{x}{Mx}$. When $x \in \{\pm1\}^n$ we have that $\norm{x}_2^2 = n$ and the claim follows from the Courant-Fischer theorem.
    \end{proof}

    To bound the spectrum of a random Gram matrix, we use a powerful result of Adamczak, Litvak, Pajor, and Tomczak-Jaegermann \cite{Adamczak11} on the isometry constant. Given a matrix $A \in \R^{m \times n}$, its isometry constant (of order $k$) is defined as the smallest number $\delta_k = \delta_k(A)$ such that
    \begin{equation}
        \label{eq:IsometryConstant}
        (1 - \delta_k)\norm{x}_2^2 \leq \norm{Ax}_2^2 \leq (1 + \delta_k)\norm{x}_2^2
    \end{equation}
    holds for all vectors $x \in \R^n$ with at most $k$ nonzero entries.

    \begin{lemma}
        \label{lem:GramSpectrum}
        Suppose that $n \leq d$. Let $X_1,\ldots,X_n$ be independent centered random vectors in $\R^d$ with $\psi_1$-norm bounded by $\psi$, and let $M$ denote the Gram matrix of $X_1,\ldots,X_n$. There exist constants $C_1,C_2 > 0$ such that, for any $\theta \in (0,1)$, holds
        \begin{equation*}
            d(1 - \delta) \leq \lambda_{\min}(G) \leq \lambda_{\max}(G) \leq d(1 + \delta)
        \end{equation*}
        with probability at least
        \begin{equation}
            \label{eq:GramSpectrumProbability}
            1 - \exp{-C_2\sqrt{n}\log(e\sqrt{d/n})} - 2\P{\max_{i=1,\ldots,n} \abs{\frac{1}{d}\norm{X_i}_2^2 - 1} \geq \theta},
        \end{equation}
        where $\delta$ is given by
        \begin{equation}
            \label{eq:GramSpectrumConstant}
            \delta = C_1(\psi + \sqrt{1 + \theta})^2\sqrt{\frac{n}{d}}\log(e\sqrt{d/n}) + \theta.
        \end{equation}
    \end{lemma}

    \begin{proof}
        Let $A$ denote the $d \times n$ random matrix with columns $X_1,\ldots,X_n$ so that $M = A^TA$. Assume that $\delta_n(\frac{1}{\sqrt{d}}A) \leq \delta$ holds with probability $p > 0$, where $\delta$ is defined by \eqref{eq:GramSpectrumConstant}. The latter event implies that $(1 - \delta)d\norm{x}_2^2 \leq \norm{Ax}_2^2 \leq d(1 + \delta)\norm{x}_2^2$ for all $x \in \R^n$. Consequently, by the Courant-Fischer theorem, we have that
        \begin{equation*}
            \lambda_{\min}(M) = \min_{\norm{x}_2 = 1} \inp{x}{Mx} = \min_{\norm{x}_2 = 1} \norm{Ax}_2^2 \geq d(1 - \delta)
        \end{equation*}
        and similarly $\lambda_{\max}(M) \leq d(1 + \delta)$, with probability at least $p$. From the estimate on the isometry constant of $\frac{1}{\sqrt{d}}A$ in Theorem 3.2 of \cite{Adamczak11}, we obtain
        \begin{align*}
            p \geq 1 &- \exp{-C_2\sqrt{n}\log(e\sqrt{d/n})} - \P{\max_{i=1,\ldots,n} \norm{X_i}_2 \geq \sqrt{(1 + \theta)d}} \\ &- \P{\max_{i=1,\ldots,n} \abs{\frac{1}{d}\norm{X_i}_2^2 - 1} \geq \theta}
        \end{align*}
        and noting that $\norm{X_i}_2 \geq \sqrt{(1 + \theta)d}$ implies $\abs{\norm{X_i}_2^2/d - 1} \geq \theta$ yields $p \geq \eqref{eq:GramSpectrumProbability}$.
    \end{proof}

    \subsection{Upper bound via vectorization}

    Our approach to an upper bound will ultimately take a similar form to the above spectral lower bound, but will be achieved by a quite different analysis.
    A random vector $X$ in $\R^m$ is called \emph{$\sigma$-subgaussian} if for all $y \in \R^m$ holds
    \begin{equation*}
        \E\exp{\inp{X}{y}} \leq \exp{\frac{\sigma^2}{2}\norm{y}_2^2}.
    \end{equation*}
    When $X$ is centered this is equivalent to an $O(\sigma)$ bound on the $\psi_2$-norm. A line of work by Bansal, Dadush, Garg, Lovett and Nikolov \cite{Bansal19a,Dadush16}, whose goal was an algorithmic version of Banaszczyk's method \cite{Banaszczyk98}, led to the following result.

    \begin{theorem}[Theorem 1.4 of \cite{Bansal19a}]
        \label{thm:SubgaussianSigning}
        Let $v_1,\ldots,v_n \in \R^m$ be vectors of Euclidean norm at most one. Then there exists a distribution of random signs $x_1,\ldots,x_n \in \{\pm1\}$ such that $\sum_{i=1}^n x_iv_i$ is $\sigma$-subgaussian, for some absolute constant $\sigma > 0$.
    \end{theorem}

    Perhaps surprisingly, this vector-valued result is already enough to derive powerful bounds on matrix discrepancy, because a standard technique allows us to control the spectral norm of a random matrix whose vectorization is subgaussian.

    \begin{lemma}
        \label{lem:SubgaussianVectorization}
        Suppose that $M$ is an $m \times m$ random symmetric matrix whose symmetric vectorization $\symvec(M)$ is $\sigma$-subgaussian. Then
        \begin{equation*}
            \P{\norm{M} \geq 4\sigma\sqrt{m}} \leq \exp{-m}.
        \end{equation*}
    \end{lemma}

    \begin{proof}
        By Chernoff's inequality, for $\lambda > 0$ we have
        \begin{align*}
            \P{y^TMy \geq t} &\leq \E\exp{\lambda y^TMy}\exp{-\lambda t},
            \intertext{and since $y^TMy = \inp{yy^T}{M}$ the subgaussian assumption gives}
            &\leq \exp{\frac{\sigma^2}{2}\lambda^2 - \lambda t}.
        \end{align*}
        Taking $\lambda = \frac{t}{\sigma^2}$ yields
        \begin{equation*}
            \P{y^TMy \geq t} \leq \exp{-\frac{t^2}{2\sigma^2}}.
        \end{equation*}
        Let $\Sigma$ be a $\frac{1}{2}$-net of the Euclidean sphere. Then Lemma 2.3.2 of \cite{Tao12} shows that
        \begin{align*}
            \P{\norm{M} \geq C\sqrt{m}} &\leq \P{\max_{y \in \Sigma} y^TMy \geq \frac{C}{2}\sqrt{m}} \leq \exp{-\frac{C^2}{8\sigma^2}}\abs{\Sigma}
            \intertext{and by Lemma 2.3.4 of \cite{Tao12} there is a $1/2$-net of size $\abs{\Sigma} \leq 3^m$ so that}
            &\leq \exp{\left(\log{3} - \frac{C^2}{8\sigma^2}\right)m}.
        \end{align*}
        Choosing $C = 4\sigma$ gives the desired result.
    \end{proof}

    \begin{lemma}
        \label{lem:UpperBound}
        Let $A_1,\ldots,A_n$ be $m \times m$ symmetric matrices. Then
        \begin{equation*}
            \disc(A_1,\ldots,A_n) = O(\sqrt{m}\max_{i=1,\ldots,n}\norm{A_i}_F).
        \end{equation*}
    \end{lemma}

    \begin{proof}
        The symmetric vectorizations $\symvec(A_1),\ldots,\symvec(A_n)$ have Euclidean norm at most $\max_{i=1,\ldots,n}\norm{A_i}_F$. Via a scaling argument \cref{thm:SubgaussianSigning} gives random signs $x_1,\ldots,x_n \in \{\pm1\}$ such that the symmetric vectorization of $M = \sum_{i=1}^n x_iA_i$ is $O(\max_{i=1,\ldots,n}\norm{A_i}_F)$-subgaussian. Then \cref{lem:SubgaussianVectorization} implies the claim.
    \end{proof}

    \subsection{Discrepancy of general random matrices}

    Now, we combine the lower and upper bounds to deduce \cref{thm:GeneralDiscrepancy}.

    \begin{proof}[Proof of \cref{thm:GeneralDiscrepancy}]
        Denote $d \colonequals m(m + 1)/2$ and consider the independent centered random vectors $X_i \colonequals \symvec(A_i)$ in $\R^d$ for $i = 1,\ldots,n$. Note that
        \begin{equation*}
            \max_{i=1,\ldots,n} \norm{X_i}_{\psi_1} = \max_{i=1,\ldots,n} \norm{A_i}_{\psi_1} \leq \psi.
        \end{equation*}
        Applying \cref{lem:GramSpectrum} shows that with probability at least \eqref{eq:GramSpectrumProbability} holds
        \begin{equation*}
            \lambda_{\min}(M) \geq d(1 - \delta),
        \end{equation*}
        where $M$ is the Gram matrix of $X_1,\ldots,X_n$ and $\delta$ is defined as in \eqref{eq:GramSpectrumConstant}. For $n \ll m^2$ holds $\delta \ll 1$ and \cref{lem:GramBound} gives the lower bound
        \begin{equation*}
            \disc(A_1,\ldots,A_n) \geq \sqrt{\frac{n}{m}\lambda_{\min}(M)} \geq \Omega(\sqrt{nm})
        \end{equation*}
        with probability at least \eqref{eq:GramSpectrumProbability}. Since $\norm{X_i}_2 = \norm{A_i}_F$ for $i = 1,\ldots,n$, the assumed norm concentration \eqref{eq:NormConcentration} and a union bound yield
        \begin{equation*}
            \eqref{eq:GramSpectrumProbability} \geq 1 - \exp{-\Omega(\sqrt{n})} - 2\P{\max_{i=1\ldots,n}\abs{\frac{1}{m^2}\norm{A_i}_F^2 - 1} \geq 1} \geq 1 - o(1) - o\of{\frac{1}{m^2}}n,
        \end{equation*}
        which tends to one due to the assumption $n \ll m^2$. The upper bound follows immediately from \cref{lem:UpperBound} and the norm concentration.
    \end{proof}

    \subsection{Discrepancy of Wishart matrices}

    The goal of this subsection is to prove \cref{thm:WishartDiscrepancy}. As mentioned earlier, we cannot apply \cref{thm:GeneralDiscrepancy} directly because Wishart matrices are not centered. Note that $\E{W} = rI_m$ for an $m \times m$ Wishart matrix $W$ of rank $r \leq m$. Instead, we work with the centered version $\overline{W} \colonequals W - rI_m$ and recover the results for the original version $W$ afterwards. We begin by stating two preliminary results. Firstly, a concentration inequality for the squared Frobenius norm of Wishart matrices and, secondly, a bound on the $\psi_1$-norm of centered Wishart matrices. The proofs of these results can be found in \cref{sec:Appendix}.

    \begin{lemma}
        \label{lem:WishartConcentration}
        Let $W$ be an $m \times m$ Wishart matrix of rank $r \leq m$. For $r \ll m$ holds
        \begin{equation*}
            \P{\abs{\frac{1}{rm^2}\norm{W}_F^2 - 1} \leq \frac{1}{2}} \geq 1 - \exp{-\Omega(m)}.
        \end{equation*}
        The same statement also applies if $W$ is replaced by its centered version $\overline{W}$.
    \end{lemma}

    \begin{lemma}
        \label{lem:WishartNorm}
        Let $W$ be an $m \times m$ Wishart matrix of rank $r \leq m$. Then
        \begin{equation*}
            \norm{\overline{W}}_{\psi_1} \leq O(\sqrt{r}).
        \end{equation*}
    \end{lemma}

    \begin{proof}[Proof of \cref{thm:WishartDiscrepancy}]
        By \cref{lem:WishartConcentration} and a union bound argument, we conclude that the maximum of $\norm{A_i}_F$ over $i = 1,\ldots,n$ is less than $O(\sqrt{r}m)$ with probability at least $1 - n\exp{-\Omega(m)}$, which tends to one due to our assumption $n \ll m^2$. Then, the desired upper bound follows immediately from \cref{lem:UpperBound}
        \begin{equation*}
            \lim_{n \to \infty} \P{\disc(A_1,\ldots,A_n) \leq O(\sqrt{rm^3})} = 1.
        \end{equation*}
        For the lower bound, we consider the Gram matrix of $\overline{W}_1,\ldots,\overline{W}_n$ scaled by $\frac{1}{\sqrt{r}}$ and denote it by $\overline{M}$. By \cref{lem:WishartNorm}, the $\psi_1$-norm of $\frac{1}{\sqrt{r}} \overline{W}_i$ is at most $O(1)$, and by \cref{lem:WishartConcentration}, the squared Frobenius norm of $\frac{1}{\sqrt{r}}\overline{W}_i$ is concentrated around $m^2$. Therefore, we can apply \cref{lem:GramSpectrum} as in the proof of \cref{thm:GeneralDiscrepancy} to conclude that
        \begin{equation}
            \label{eq:WishartProof3}
            \lambda_{\min}(\overline{M}) \geq \Omega(m^2)
        \end{equation}
        asymptotically almost surely. Let $M$ denote the Gram matrix of $W_1,\ldots,W_n$ scaled by $\frac{1}{\sqrt{r}}$. The relationship between the entries of $M$ and $\overline{M}$ is given by
        \begin{align*}
            \overline{M}_{ij} &= M_{ij} - \inp{G_iG_i^T}{I_m} - \inp{I_m}{G_jG_j^T} + rm \\ &= M_{ij} - \norm{G_i}_F^2 - \norm{G_j}_F^2 + rm \\ &= M_{ij} - (\norm{G_i}_F^2 - \E\norm{G_i}_F^2) - (\norm{G_j}_F^2 - \E\norm{G_j}_F^2) - rm,
        \end{align*}
        where in the last line we used that $\E\norm{G_i}_F^2 = rm$. In matrix terms, we have that
        \begin{equation*}
            \overline{M} = M - rm(y1_n^T + 1_ny^T + 1_n1_n^T) = M - rm(y + 1_n)(y + 1_n)^T + rmyy^T,
        \end{equation*}
        where $1_n$ denotes the $n$-dimensional all-ones vector and $y$ is the $n$-dimensional vector with components $y_i \colonequals \frac{1}{rm}(\norm{G_i}_F^2 - \E\norm{G_i}_F^2)$ for $i = 1,\ldots,n$. Applying Weyl's inequality, we see that
        \begin{equation}
            \label{eq:WishartProof4}
            \lambda_{\min}(M) \geq \lambda_{\min}(\overline{M}) - rm\lambda_{\max}(yy^T) = \lambda_{\min}(\overline{M}) - \frac{1}{rm}\sum_{i=1}^n y_i^2,
        \end{equation}
        where we used that $\lambda_{\min}((y + 1_n)(y + 1_n)^T) \geq 0$. A straightforward calculation shows that $\E{y_i^2} = 2rm$. By the law of large numbers, we conclude that $\sum_{i=1}^n y_i^2 = O(nrm)$ asymptotically almost surely. Combining this with \eqref{eq:WishartProof3} and \eqref{eq:WishartProof4} yields
        \begin{equation*}
            \lambda_{\min}(M) \geq \Omega(m^2) - O(n) \geq \Omega(m^2)
        \end{equation*}
        if $n \ll m^2$. Keeping in mind that $M$ corresponds to the Gram matrix of $W_1,\ldots,W_n$ scaled by $\frac{1}{\sqrt{r}}$, it follows from \cref{lem:GramBound} that
        \begin{equation*}
            \lim_{n \to \infty} \P{\disc(A_1,\ldots,A_n) \geq \Omega(\sqrt{rnm})} = 1.
            \qedhere
        \end{equation*}
    \end{proof}

    \section{Analysis of the MHC algorithm}
    \label{sec:GeneralAnalysis}

    In this section, we carry out the analysis of \cref{alg:MatrixHyperbolicCosine} for random inputs, which leads to \cref{thm:GeneralAnalysis}. We begin with some preliminaries on transcendental matrix functions. The matrix exponential of $X \in \R^{m \times m}$ is defined as
    \begin{equation*}
        \exp{X} \colonequals \sum_{k=0}^\infty \frac{X^k}{k!}.
    \end{equation*}
    If $X$ is symmetric, it can be computed via the eigendecomposition $X = QDQ^T$, where $Q$ is an orthonormal matrix with the eigenvectors as columns and $D$ is a diagonal matrix of the eigenvalues. Then $\exp{X} = Q\exp{D}Q^T$, where $\exp{D}$ is a diagonal matrix with entries $\exp{\lambda_i(X)}$. In particular, this shows that the spectrum of $\exp{X}$ is given by
    \begin{equation}
        \label{eq:MatrixExponential}
        \lambda_i(\exp{X}) = \exp{\lambda_i(X)}.
    \end{equation}
    The matrix hyperbolic sine and cosine of $X \in \R^{m \times m}$ are defined as
    \begin{equation*}
        \sinh(X) \colonequals \frac{\exp{X} - \exp{-X}}{2}, \quad \cosh(X) \colonequals \frac{\exp{X} + \exp{-X}}{2}.
    \end{equation*}
    It is not hard to verify that many properties for the scalar hyperbolic functions also apply to the matrix-valued counterparts. We summarize the properties that are essential for our analysis in the following lemma.

    \begin{lemma}
        \label{lem:HyperbolicProperties}
        Let $X \in \R^{m \times m}$ be a symmetric matrix. Then the following hold
        \begin{enumerate}[label=(\roman*)]
            \item $\norm{\cosh(X) - I_m} = \cosh(\norm{X}) - 1$, \label{lem:HyperbolicProperties1}
            \item $\norm{\sinh(X) - X} \leq \norm{X}^3$ if $\norm{X} \leq 1$, \label{lem:HyperbolicProperties2}
            \item $\norm{\sinh(X)}_* = \sum_{i=1}^m \abs{\sinh(\lambda_i(X))}$, \label{lem:HyperbolicProperties3}
            \item $\tr{\cosh(X)} = \sum_{i=1}^m \cosh(\lambda_i(X))$. \label{lem:HyperbolicProperties4}
        \end{enumerate}
    \end{lemma}

    The first two properties follow from the power series expansions of the matrix hyperbolic functions, and the last two properties follow from \eqref{eq:MatrixExponential} and the linearity of trace; we omit the details. The Golden-Thompson inequality
    \begin{equation}
        \label{eq:GoldenThompsonInequality}
        \tr\exp{X + Y} \leq \inp{\exp{X}}{\exp{Y}}
    \end{equation}
    addresses the main property of scalar transcendental functions that is not inherited by their matrix versions, namely that $\exp(X + Y) \neq \exp(X)\exp(Y)$ in general.
    It plays a key role in our analysis of \cref{alg:MatrixHyperbolicCosine}, as it allows us to bound the value of the potential function we work with when it is applied to sums of matrices.

    \begin{lemma}
        \label{lem:GoldenThompsonInequality}
        For symmetric matrices $X,Y \in \R^{m \times m}$ holds
        \begin{equation*}
            \tr\cosh(X + Y) \leq \inp{\cosh(X)}{\cosh(Y)} + \inp{\sinh(X)}{\sinh(Y)}.
        \end{equation*}
    \end{lemma}

    \begin{proof}
        We expand and apply the Golden-Thompson inequality \eqref{eq:GoldenThompsonInequality} to obtain
        \begin{align*}
            \tr\cosh(X + Y) &= \frac{1}{2}(\tr\exp{X + Y} + \tr\exp{-X - Y}) \\ &\leq \frac{1}{2}(\inp{\exp{X}}{\exp{Y}} + \inp{\exp{-X}}{\exp{-Y}}),
        \end{align*}
        which yields the claim after some simple algebraic manipulations.
    \end{proof}

    We note that in the scalar case where $m = 1$, the above inequality holds as an equality and is the standard sum rule for the hyperbolic cosine. For our analysis, we further rely on the Hölder inequality for Schatten norms
    \begin{equation}
        \label{eq:HoelderInequality}
        \inp{X}{Y} \leq \norm{X}\norm{Y}_*,
    \end{equation}
    and the inner product inequality
    \begin{equation}
        \label{eq:MatrixInequality}
        \inp{X}{Y} \leq \inp{X}{Z}
    \end{equation}
    for symmetric matrices $X,Y,Z \in \R^{m \times m}$ with $X \succeq 0$ and $Z \succeq Y$, see Lemma 2.2 in \cite{Tsuda05}. Here and in the following, we denote by $\succeq$ the Loewner order (defined as $X \succeq Y$ if and only if $X - Y$ is positive semidefinite).

    \subsection{Drift analysis for the potential}

    Let us denote the potential function in \cref{alg:MatrixHyperbolicCosine} by $\Phi(X) \colonequals \tr\cosh(\alpha X)$ for $X \in \R^{m \times m}$, where $\alpha > 0$ is the parameter of the algorithm whose value will be determined later. Throughout this section, we assume that $(A_t)_{t \in \N}$ is a sequence of independent copies of a random matrix $A$ with $\norm{A} \leq 1$ that satisfies Conditions \eqref{eq:MatrixAntiConcentration} and \eqref{eq:Unbiasedness}, and let $(x_t)_{t \in \N}$ denote the signs generated by \cref{alg:MatrixHyperbolicCosine} when run on input $(A_t)_{t \in \N}$. Note that $M_t \colonequals \sum_{i=1}^t x_iA_i$ for $t \in \N$ defines a Markov chain on the state space of symmetric matrices, and the potential $\Phi$ defines a Lyapunov function that maps each state to a real number, giving rise to the real-valued random process $\Phi_t \colonequals \Phi(M_t)$ for $t \in \N$. By convention, we let $M_0$ be the $m \times m$ zero matrix so that $\Phi_0 \colonequals \Phi(M_0) = m$. In the following lemma, we examine the \emph{drift} of $(\Phi_t)_{t \in \N}$ at time $t \in \N$, that is, the random variable $\E(\Phi_t - \Phi_{t-1} \mid A_1,\ldots,A_{t-1})$. We abbreviate the conditional expectation with respect to $A_1,\ldots,A_{t-1}$ by $\E[t]$ so that $\E[t]{\Phi_t - \Phi_{t-1}} = \E{\Phi_t - \Phi_{t-1} \mid A_1,\ldots,A_{t-1}}$.

    \begin{lemma}
        \label{lem:BoundedDrift}
        Suppose that $\alpha \ll (rm)^{-1/2}$. Then for $t \in \N$,
        \begin{equation*}
            \E[t]{\Phi_t - \Phi_{t-1} \mid \Phi_{t-1} = x} \leq
            \begin{cases}
                O(m^{-1}) & \text{if}~x \leq 2m, \\
                -\Omega(m^{-2})x & \text{if}~x \geq 2m,
            \end{cases}
        \end{equation*}
        where the implicit constants depend only on the parameters $\eta$ and $\theta$.
    \end{lemma}

    \begin{proof}
        Throughout this proof, we drop the condition in the expectation $\E[t]{\Phi_t - \Phi_{t-1} \mid \Phi_{t-1} = x}$ and simply write $\E[t]{\Phi_t - \Phi_{t-1}}$ instead; we treat $\Phi_{t-1}$ as a deterministic variable.

        \textit{Step 1: Breaking up the increment.} We analyse the increment as
        \begin{align*}
            \Phi_t - \Phi_{t-1} &= \tr\cosh(\alpha M_{t-1} + \alpha x_tA_t) - \tr\cosh(\alpha M_{t-1})
            \intertext{apply \cref{lem:GoldenThompsonInequality} to obtain}
            &\leq \inp{\cosh(\alpha M_{t-1})}{\cosh(\alpha x_tA_t) - I_m} + \inp{\sinh(\alpha M_{t-1})}{\sinh(\alpha x_tA_t)}
            \intertext{and use parity properties of the hyperbolic functions}
            &= \inp{\cosh(\alpha M_{t-1})}{\cosh(\alpha A_t) - I_m} + x_t\inp{\sinh(\alpha M_{t-1})}{\sinh(\alpha A_t)}.
        \end{align*}
        From this we conclude that the choice of $x_t \in \{\pm1\}$ by \cref{alg:MatrixHyperbolicCosine} achieves
        \begin{equation*}
            \Phi_t - \Phi_{t-1} \leq \underbrace{\inp{\cosh(\alpha M_{t-1})}{\cosh(\alpha A_t) - I_m}}_{\equalscolon T_1} - \underbrace{\abs{\inp{\sinh(\alpha M_{t-1})}{\sinh(\alpha A_t)}}}_{\equalscolon T_2}.
        \end{equation*}
        We now investigate the expectations of the two terms $T_1$ and $T_2$ separately.

        \textit{Step 2: Bounding the first term.} The power series expansion of the hyperbolic cosine shows that $\cosh(\alpha A_t) - I_m = (\cosh(\alpha A_t) - I_m)P_{\row(A_t)}$. Since
        \begin{equation*}
            \norm{\cosh(\alpha A_t) - I_m}P_{\row(A_t)} \succeq (\cosh(\alpha A_t) - I_m)P_{\row(A_t)}
        \end{equation*}
        and $\cosh(\alpha M_{t-1}) \succeq 0$, we can apply \eqref{eq:MatrixInequality} to obtain
        \begin{align*}
            \E[t]{T_1} &\leq \E[t]{\inp{\cosh(\alpha M_{t-1})}{\norm{\cosh(\alpha A_t) - I_m}P_{\row(A_t)}}} \\
            &\leq (\cosh(\alpha) - 1)\inp{\cosh(\alpha M_{t-1})}{\E P_{\row(A_t)}} \tag{by Property \ref{lem:HyperbolicProperties1}} \\
            &\leq (\cosh(\alpha) - 1)\tr\cosh(\alpha M_{t-1})\norm{\E P_{\row(A_t)}} \tag{by Hölder's inequality \eqref{eq:HoelderInequality}} \\
            &\leq (\cosh(\alpha) - 1)\theta\frac{r}{m}\Phi_{t-1} \tag{by Condition \eqref{eq:Unbiasedness}} \\
            &\leq \alpha^2\theta\frac{r}{m}\Phi_{t-1}. \tag{as $\cosh(x) - 1 \leq x^2$}
        \end{align*}
        By simply discarding the second term, we obtain in the case $\Phi_{t-1} \leq 2m$ that
        \begin{equation*}
            \E[t]{\Phi_t - \Phi_{t-1}} \leq 2\alpha^2\theta r,
        \end{equation*}
        which is of order $O(m^{-1})$ when $\alpha = O((rm)^{-1/2})$. In the case $\Phi_{t-1} \geq 2m$, we must also take the second term into account.

        \textit{Step 3: Bounding the second term.} We apply the reverse triangle inequality
        \begin{equation*}
            T_2 \geq \underbrace{\abs{\inp{\sinh(\alpha M_{t-1})}{\alpha A_t}}}_{\equalscolon S_1} - \underbrace{\abs{\inp{\sinh(\alpha M_{t-1})}{\sinh(\alpha A_t) - \alpha A_t}}}_{\equalscolon S_2}
        \end{equation*}
        to decompose $T_2$ into two terms, which we proceed to bound individually. For the term $S_1$, we use Condition \eqref{eq:MatrixAntiConcentration} to get
        \begin{align*}
            \E[t]{S_1} &\geq \alpha\eta\sqrt{\frac{r}{m^3}} \norm{\sinh(\alpha M_{t-1})}_* \\
            &= \alpha\eta\sqrt{\frac{r}{m^3}} \sum_{i=1}^m \abs{\sinh(\lambda_i(\alpha M_{t-1}))} \tag{by Property \ref{lem:HyperbolicProperties3}} \\
            &\geq \alpha\eta\sqrt{\frac{r}{m^3}} \sum_{i=1}^m (\cosh(\lambda_i(\alpha M_{t-1})) - 1) \tag{as $\abs{\sinh(x)} \geq \cosh(x) - 1$} \\
            &= \alpha\eta\sqrt{\frac{r}{m^3}}(\Phi_{t-1} - m). \tag{by Property \ref{lem:HyperbolicProperties4}}
        \end{align*}
        For the term $S_2$, we argue as above. The power series expansion of the hyperbolic sine shows that $\sinh(\alpha A_t) - \alpha A_t = (\sinh(\alpha A_t) - \alpha A_t)P_{\row(A_t)}$. Since
        \begin{equation*}
            \norm{\sinh(\alpha A_t) - \alpha A_t}P_{\row(A_t)} \succeq (\sinh(\alpha A_t) - \alpha A_t)P_{\row(A_t)}
        \end{equation*}
        and $\abs{\sinh(\alpha M_{t-1})} \succeq 0$, we can apply \eqref{eq:MatrixInequality} to obtain
        \begin{align*}
            \E[t]{S_2} &\leq \E[t]{\inp{\abs{\sinh(\alpha M_{t-1})}}{\norm{\sinh(\alpha A_t) - \alpha A_t}P_{\row(A_t)}}} \\
            &\leq \alpha^3\inp{\abs{\sinh(\alpha M_{t-1})}}{\E P_{\row(A_t)}} \tag{by Property \ref{lem:HyperbolicProperties2}} \\
            &\leq \alpha^3\norm{\sinh(\alpha M_{t-1})}_*\norm{\E P_{\row(A_t)}} \tag{by Hölder's inequality \eqref{eq:HoelderInequality}} \\
            &\leq \alpha^3\theta\frac{r}{m}\norm{\sinh(\alpha M_{t-1})}_* \tag{by Condition \eqref{eq:Unbiasedness}} \\
            &\leq \alpha^3\theta\frac{r}{m}\Phi_{t-1}. \tag{as $\abs{\sinh(x)} \leq \cosh(x)$}
        \end{align*}
        \textit{Step 4: Putting everything together.} Putting all bounds together, we find that
        \begin{align*}
            \E[t]{\Phi_t - \Phi_{t-1}} &\leq (\alpha^2 + \alpha^3)\theta\frac{r}{m}\Phi_{t-1} - \alpha\eta\sqrt{\frac{r}{m^3}}(\Phi_{t-1} - m),
            \intertext{which in the case $\Phi_{t-1} \geq 2m$ can be estimated by}
            &\leq \left(2\alpha^2\theta\frac{r}{m} - \frac{1}{2}\alpha\eta\sqrt{\frac{r}{m^3}}\right)\Phi_{t-1}.
        \end{align*}
        The latter is bounded above by $-\frac{\eta^2}{32\theta m^2}\Phi_{t-1} = -\Omega(m^{-2})\Phi_{t-1}$ when $\alpha \leq \frac{\eta}{8\theta\sqrt{rm}}$, completing the proof.
    \end{proof}

    Now, we transform the knowledge about the drift of $(\Phi_t)_{t \in \N}$ into bounds on the discrepancy and prefix discrepancy, respectively, achieved by \cref{alg:MatrixHyperbolicCosine}.

    \begin{proof}[Proof of \cref{thm:GeneralAnalysis}]
        By the law of total expectation, we have that
        \begin{equation}
            \label{eq:GeneralAnalysisProof1}
            \E{\Phi_t} = \E{\E[t]{\Phi_t}} = \int_0^\infty (\E[t]{\Phi_t - \Phi_{t-1} \mid \Phi_{t-1} = x} + x) d\P{\Phi_{t-1} = x}.
        \end{equation}
        Using the bounds on the drift of $(\Phi_t)_{t \in \N}$ in \cref{lem:BoundedDrift}, we conclude that
        \begin{equation}
            \label{eq:GeneralAnalysisProof2}
            \int_0^{2m} (\E[t]{\Phi_t - \Phi_{t-1} \mid \Phi_{t-1} = x} + x) d\P{\Phi_{t-1} = x} \leq O(m)
        \end{equation}
        and
        \begin{equation}
            \label{eq:GeneralAnalysisProof3}
            \int_{2m}^\infty (\E[t]{\Phi_t - \Phi_{t-1} \mid \Phi_{t-1} = x} + x) d\P{\Phi_{t-1} = x} \leq (1 - \varepsilon) \E{\Phi_{t-1}}
        \end{equation}
        for some $\varepsilon = \Omega(m^{-2})$. Combining the bounds \eqref{eq:GeneralAnalysisProof2}, \eqref{eq:GeneralAnalysisProof3} with \eqref{eq:GeneralAnalysisProof1} yields that $\E{\Phi_t} \leq O(m) + (1 - \varepsilon)\E{\Phi_{t-1}}$, and inductively follows that
        \begin{equation*}
            \E{\Phi_t} \leq O(m) \sum_{k=0}^t (1 - \varepsilon)^k \leq O(m) \sum_{k=0}^\infty (1 - \varepsilon)^k = \frac{O(m)}{\varepsilon} = O(m^3).
        \end{equation*}
        Thus, by Markov's inequality,
        \begin{equation}
            \label{eq:GeneralAnalysisProof4}
            \P{\Phi_t \leq O(m^4)} \geq 1 - m^{-1}.
        \end{equation}
        Using Property \ref{lem:HyperbolicProperties1} of \cref{lem:HyperbolicProperties}, we obtain that
        \begin{equation*}
            \exp{\norm{X}} - 1 \leq \cosh(\norm{X}) - 1 = \norm{\cosh(X) - I_m} \leq \tr(\cosh(X) - I_m),
        \end{equation*}
        and thereby conclude that
        \begin{equation}
            \label{eq:GeneralAnalysisProof5}
            \norm{X} \leq \frac{1}{\alpha}\log{\Phi(X)} = O(\sqrt{rm})\log{\Phi(X)}.
        \end{equation}
        Therefore, on the event $\Phi_n \leq O(m^4)$, we have that $\norm{M_n} \leq O(\sqrt{rm}\log{m})$. Then \eqref{eq:GeneralAnalysisProof4} implies Part \ref{thm:GeneralAnalysis2}. To prove Part \ref{thm:GeneralAnalysis1}, we note that the drift of $(\Phi_t)_{t \in \N}$ is uniformly bounded by $O(m^{-1})$, according to \cref{lem:BoundedDrift}. So we have that
        \begin{equation*}
            \E{\Phi_t} = \E{\E[t]{\Phi_t}} \leq \E{\Phi_{t-1}} + O(m^{-1}),
        \end{equation*}
        and inductively it follows that $\E{\Phi_t} = O(tm^{-1})$. Thus, by Markov's inequality
        \begin{equation*}
            \P{\Phi_t \leq O(n^3m^{-1})} \geq 1 - n^{-2}
        \end{equation*}
        for $t = 1,\ldots,n$ and taking a union bound yields
        \begin{equation*}
            \P{\max_{t=1,\ldots,n} \Phi_t \leq O(n^3m^{-1})} \geq 1 - n^{-1}.
        \end{equation*}
        Similarly as above, this and \eqref{eq:GeneralAnalysisProof5} lead to Part \ref{thm:GeneralAnalysis1}.
    \end{proof}

    \section{Applications of the MHC algorithm}
    \label{sec:Applications}

    In this section, we specialize \cref{thm:GeneralAnalysis} to different random matrix distributions, which amounts to verifying Condition \eqref{eq:MatrixAntiConcentration} and \eqref{eq:Unbiasedness}. While checking the unbiasedness property \eqref{eq:Unbiasedness} is usually straightforward, verifying the matrix anti-concentration inequality \eqref{eq:MatrixAntiConcentration} requires more work. We begin by introducing a tool that provides a sufficient condition for \eqref{eq:MatrixAntiConcentration} to hold. A random vector $X$ that takes values in $\R^d$ is said to satisfy the \emph{Khintchine anti-concentration inequality} with parameter $\eta > 0$ if for all $y \in \R^d$ holds
    \begin{equation}
        \label{eq:KhintchineAntiConcentration}
        \E\abs{\inp{X}{y}} \geq \eta\norm{y}_2.
    \end{equation}
    For an $m \times m$ random matrix $A$, we observe that if its symmetric vectorization $\symvec(A)$ satisfies \eqref{eq:KhintchineAntiConcentration}, then $A$ satisfies the matrix anti-concentration inequality \eqref{eq:MatrixAntiConcentration} with parameter $\eta/\sqrt{m}$. This follows immediately from the definition along with the inequality $\norm{Y}_F \geq \norm{Y}_*/\sqrt{m}$ for $Y \in \R^{m \times m}$.

    \begin{lemma}
        \label{lem:KhintchineAntiConcentration}
        Let $X$ be a random vector in $\R^d$ that satisfies \eqref{eq:KhintchineAntiConcentration} with parameter $\eta$, and assume that the spectral norm of its covariance matrix is bounded by $\rho > 0$. For any event $E$ with $1 - \P{E} < \frac{\eta^2}{\rho}$, the random vector $X$ conditioned on $E$ (or $X$ trunacted to $E$) satisfies \eqref{eq:KhintchineAntiConcentration} with parameter $\eta' = \eta - \sqrt{\rho(1 - \P{E})} > 0$.
    \end{lemma}

    \begin{proof}
        For $y \in \R^d$ holds $\E\abs{\inp{X}{y}} \geq \eta\norm{y}_2$ by assumption. Then, we have
        \begin{align*}
            \E{\abs{\inp{X}{y}}\ind{E}} &\geq \eta\norm{y}_2 - \E{\abs{\inp{X}{y}}\ind{E^c}}
            \intertext{applying the Cauchy-Schwarz inequality yields}
            &\geq \eta\norm{y}_2 - \sqrt{\E{\inp{X}{y}^2}\P{E^c}}
            \intertext{and bounding the remaining expectation as $y^T\cov(X)y \leq \rho\norm{y}_2^2$ gives}
            &\geq \left(\eta - \sqrt{\rho(1 - \P{E})}\right)\norm{y}_2,
        \end{align*}
        which shows the claim for $X$ truncated to $E$. Since the expectation of $\abs{\inp{X}{y}}$ given $E$ is always greater than $ \E{\abs{\inp{X}{y}}\ind{E}}$, the conditional version follows.
    \end{proof}

    \subsection{Hypercontractive Wigner ensemble}

    In this subsection, we consider Wigner matrices with hypercontractive entries. A random variable $X$ is said to be \emph{$\kappa$-hypercontractive} for $0 < \kappa < 1$ if its fourth moment is finite and satisfies
    \begin{equation}
        \label{eq:Hypercontractivity}
        \kappa^4\E{X^4} \leq \E{X^2}^2.
    \end{equation}
    In other words, a random variable is hypercontractive if its fourth moment is small compared to its second moment. For a list of basic properties of hypercontractive random variables we refer to Gopalan, O’Donnell, Wu and Zuckerman\cite{Gopalan10}, even though they work with a stricter notion of hypercontractivity. Here we are mostly interested in an anti-concentration property of hypercontractive random variables.

    \begin{lemma}[Proposition III.6 of \cite{Gopalan10}]
        \label{lem:HypercontractiveAntiConcentration}
        If $X$ is a $\kappa$-hypercontractive random variable, then for $0 < t < 1$ holds
        \begin{equation*}
            \P{\abs{X} \geq t\E{X^2}^{1/2}} \geq \kappa^4(1 - t^2)^2.
        \end{equation*}
    \end{lemma}

    \begin{lemma}
        \label{lem:Hypercontractivity}
        Let $X$ be a random vector in $\R^d$ with independent, centered and $\kappa$-hypercontractive entries, then $X$ satisfies the Khintchine anti-concentration inequality \eqref{eq:KhintchineAntiConcentration} with parameter $\eta = \Omega(\sigma\kappa^4)$, where $\sigma^2 = \min_i \E{X_i^2}$.
    \end{lemma}

    \begin{proof}
        Let $y \in \R^d$ be an arbitrary vector. Then $\inp{X}{y}$ is $\kappa$-hypercontractive as a linear combination of $\kappa$-hypercontractive random variables. From \cref{lem:HypercontractiveAntiConcentration} follows that
        \begin{equation*}
            \P{\abs{\inp{X}{y}} \geq \frac{1}{2}\left(\E{\inp{X}{y}^2}\right)^{1/2}} \geq \Omega(\kappa^4).
        \end{equation*}
        Since the entries of $X$ are independent and centered by assumption, we have that $\inp{X}{y}^2  = \sum_{i=1}^d \E{X_i^2}y_i^2 \geq \sigma^2\norm{y}_2^2$ and therefore
        \begin{equation*}
            \E{\abs{\inp{X}{y}}} \geq \P{\abs{\inp{X}{y}} \geq \frac{\sigma}{2}\norm{y}_2} \frac{\sigma}{2}\norm{y}_2 \geq \Omega(\kappa^4\sigma)\norm{y}_2,
        \end{equation*}
        completing the proof.
    \end{proof}

    \begin{theorem}
        \label{thm:HypercontractiveWignerEnsemble}
        Let $A$ be an $m \times m$ Wigner matrix with centered hypercontractive entries whose second moment is uniformly bounded by $C_1^2/m \leq \E{A_{ij}^2} \leq C_2^2/m$ for some sufficiently small constant $C_1,C_2 > 0$. Then $A$ conditioned on $\norm{A} \leq 1$ satisfies Conditions \eqref{eq:MatrixAntiConcentration} and \eqref{eq:Unbiasedness} with $r = m$ for parameters $\eta,\theta > 0$ depending only on the distribution of its entries.
    \end{theorem}

    \begin{proof}
        When $r = m$ the unbiasedness condition \eqref{eq:Unbiasedness} holds trivially with $\theta = 1$ as orthogonal projection matrices have spectral norm at most one. It remains to check that $A$ given $\norm{A} \leq 1$ satisfies the matrix anti-concentration condition \eqref{eq:MatrixAntiConcentration} with $r = m$, that is, for all symmetric matrices $Y \in \R^{m \times m}$ holds $\E\abs{\inp{A}{Y}} \geq \Omega(1/m)\norm{Y}_*$. Using the observation made below \eqref{eq:KhintchineAntiConcentration}, it suffices to show instead that $X \colonequals \symvec(A)$ given $\norm{A} \leq 1$ fulfills the Khintchine anti-concentration condition \eqref{eq:KhintchineAntiConcentration} with parameter $\eta = \Omega(1/\sqrt{m})$. By assumption, we have that
        \begin{equation*}
            C_1^2/m \leq \E{A_{ij}^2} \leq C_2^2/m, \quad 1 \leq i,j \leq m
        \end{equation*}
        for some constants $C_1,C_2 > 0$. From \cref{lem:Hypercontractivity} follows that $X$ satisfies \eqref{eq:KhintchineAntiConcentration} with $\eta = \Omega(C_1/\sqrt{m})$. Since the covariance matrix of $X$ is a diagonal matrix with entries $\E{A_{ij}^2}$, its spectral norm is at most $C_2^2/m$ and \cref{lem:KhintchineAntiConcentration} implies that $X$ given $\norm{A} \leq 1$ satisfies \eqref{eq:KhintchineAntiConcentration} with $\eta = \Omega(C_1/\sqrt{m})$ if $\norm{A} \leq 1$ occurs with probability high enough. Due to the hypercontractivity, we have that $\E{A_{ij}^4} = O(C_2^4/m^2)$. Then, a result of Lata{\l}a \cite{Latala05} implies that $\E\norm{A} = O(C_2)$, which when combined with Markov's inequality gives the large deviation inequality $\P{\norm{A} \geq 1} = O(C_2)$. So choosing $C_2$ small enough allows us to achieve the desired probability for the event $\norm{A} \leq 1$ and finishes the proof.
    \end{proof}

    \subsection{Normalized Wishart ensemble}

    In this subsection, we consider normalized Wishart matrices to illustrate the rank dependence in Conditions \eqref{eq:MatrixAntiConcentration} and \eqref{eq:Unbiasedness}. To verify the matrix anti-concentration inequality \eqref{eq:MatrixAntiConcentration}, we rely on the following auxiliary result that follows from standard polynomial anti-concentration results, see for example Lovett \cite{Lovett10}.

    \begin{lemma}
        \label{lem:PolynomialAntiConcentration}
        Let $g$ be an $d$-dimensional Gaussian random vector and let $X$ be the vector of all degree $k$ monomials in the entries of $g$. Then $X$ satisfies \eqref{eq:KhintchineAntiConcentration} with a parameter $\eta > 0$ depending only on $k$.
    \end{lemma}

    \begin{theorem}
        \label{thm:NormalizedWishartEnsemble}
        Let $W$ be an $m \times m$ Wishart matrix of rank $r \leq m$. The normalized Wishart matrix $W/\norm{W}$ satisfies Conditions \eqref{eq:MatrixAntiConcentration} and \eqref{eq:Unbiasedness} for constant $\eta,\theta > 0$.
    \end{theorem}

    \begin{proof}
        By definition, we have $W = GG^T$ for some $m \times r$ Gaussian matrix $G$. First, we verify the unbiasedness condition \eqref{eq:Unbiasedness}. By the rotational invariance of the Gaussian distribution, we conclude that $\row(W)$ is a uniformly distributed $r$-dimensional subspace of $\R^m$. Consequently, the orthogonal projection $P_{\row(W)}$ has the same distribution as $U^TU$, where $U$ consists of the first $r$ columns of a Haar distributed $m \times m$ orthogonal matrix. From Lemma 3.3 of Meckes \cite{Meckes06} follows that
        \begin{equation*}
            \E P_{\row(W)} = \frac{r}{m}I_m,
        \end{equation*}
        and hence $W/\norm{W}$ satisfies \eqref{eq:Unbiasedness} with $\theta = 1$. Next, we check the matrix anti-concentration condition \eqref{eq:MatrixAntiConcentration}. Let $Y \in \R^{m \times m}$ be an arbitrary matrix. Note that $\norm{W} = \norm{G}^2$ and therefore truncation to the event $\norm{G} \leq 3\sqrt{m}$ leads to
        \begin{equation}
            \label{eq:NormalizedWishartProof}
            \E\abs{\frac{1}{\norm{W}}\inp{W}{X}} \geq \frac{1}{9m} \E{\abs{\inp{GG^T}{X}}\ind{\norm{G} \leq 3\sqrt{m}}}.
        \end{equation}
        Denote the columns of $G$ by $g_1,\ldots,g_r$ and let $h_i$ be the vector of all degree two monomials in the entries of $g_i$ for $i = 1,\ldots,r$. Substituting $GG^T = \sum_{i=1}^r g_ig_i^T$ yields
        \begin{equation*}
            \inp{GG^T}{Y} = \sum_{i=1}^r g_i^TYg_i = \sum_{i=1}^r \inp{h_i}{y},
        \end{equation*}
        where $y \colonequals \symvec(Y)$. From Jensen's inequality and \cref{lem:PolynomialAntiConcentration} follows that
        \begin{equation*}
            \E \abs{\inp{GG^T}{Y}} \geq \sqrt{r}\E \abs{\inp{h_i}{y}} \geq \Omega(\sqrt{r})\norm{y}_2.
        \end{equation*}
        Since $\norm{G} \leq 3\sqrt{m}$ with probability at least $1 - 2\exp{-\Omega(m)}$ by Lemma 7.3.3 of Vershynin \cite{Vershynin18} and the covariance matrix of $\symvec(GG^T)$ has spectral norm at most $O(1)$, for $m$ large enough \cref{lem:KhintchineAntiConcentration} implies that
        \begin{equation}
            \E{\abs{\inp{GG^T}{Y}}\ind{\norm{G} \leq 3\sqrt{m}}} \geq \Omega(\sqrt{r})\norm{y}_2 \geq \Omega(\sqrt{r/m})\norm{Y}_*.
        \end{equation}
        This combined with \eqref{eq:NormalizedWishartProof} yields \eqref{eq:MatrixAntiConcentration} with $\eta = \Omega(\sqrt{r/m^3})$.
    \end{proof}

    \section*{Acknowledgments}
\addcontentsline{toc}{section}{Acknowledgments}

    We thank Afonso Bandeira, David Gamarnik, and Daniel Spielman for helpful discussions.
We also thank Antoine Maillard and Peng Zhang for pointing out several errors in earlier versions of this paper, in addition to general discussions.
Lastly, we are grateful to an anonymous reviewer for pointing out the application of Lemma~\ref{lem:SubgaussianVectorization} in the context of matrix discrepancy.

    \clearpage

    \appendix

    \section{Omitted proofs}
    \label{sec:Appendix}

    \begin{proof}[Proof of \cref{lem:GaussianMatrix2}]
        Let $x_{ij}$ and $y_{ij}$ denote the entries of $X$ and $Y$, respectively. Further, let $f_{ij}$ denote the joint density of $x_{ij}$ and $y_{ij}$, that is, the density two jointly normal random variables with mean zero, variance $\sigma_{ij}^2$ (where $\sigma_{ij}^2 = 2$ if $i = j$ and $\sigma_{ij}^2 = 1$ otherwise) and correlation $\rho$. Since the entries of a GOE matrix are independent, it follows that the joint probability density function is given by
        \begin{equation*}
            \prod_{1 \leq i,j \leq m} f_{ij}(x_{ij},y_{ij}).
        \end{equation*}
        From \eqref{eq:BivariateNormalDistribution} we conclude that the diagonal entries contribute to the density by
        \begin{equation*}
            \prod_{1 \leq i = j \leq m} f_{ij}(x_{ij},y_{ij}) = \frac{1}{(4\pi\sqrt{1 - \rho^2})^m} \prod_{1 \leq i = j \leq m} \exp{-\frac{x_{ij}^2 - 2\rho x_{ij}y_{ij} + y_{ij}^2}{4(1 - \rho^2)}},
        \end{equation*}
        and the off-diagonal entries contribute to the density by
        \begin{equation*}
            \prod_{1 \leq i < j \leq m} f_{ij}(x_{ij},y_{ij}) = \frac{1}{(2\pi\sqrt{1 - \rho^2})^{m(m-1)/2}} \prod_{1 \leq i < j \leq m} \exp{-\frac{x_{ij}^2 - 2\rho x_{ij}y_{ij} + y_{ij}^2}{2(1 - \rho^2)}}.
        \end{equation*}
        Using symmetry of $X$ and $Y$, we can rewrite the product of these two terms as
        \begin{equation*}
            K_{m}^2 (1 - \rho^2)^{-m(m+1)/4} \exp{-\frac{1}{4(1 - \rho^2)}\sum_{1 \leq i,j \leq m} x_{ij}^2 - 2\rho x_{ij}y_{ij} + y_{ij}^2}.
        \end{equation*}
        The desired representation is obtained by noting that
        \begin{equation*}
            \sum_{1 \leq i,j \leq m} x_{ij}^2 - 2\rho x_{ij}y_{ij} + y_{ij}^2 = \tr(X^2 - 2\rho XY + Y^2).
            \qedhere
        \end{equation*}
    \end{proof}

    \begin{proof}[Proof of \cref{lem:GaussianEigenvalues2}]
        Let $S(\lambda)$ denote the set of all $2 \times 2$ symmetric matrices with spectrum $\{\lambda_1,\lambda_2\}$ for $\lambda \in \R^2$, and let $f$ denote the joint probability density function of $X$ and $Y$. Then, for any domain $D \subseteq \R^2_\geq \times \R^2_\geq$ the probability that the ordered eigenvalues of $X$ and $Y$ fall into $D$ is given by
        \begin{equation}
            \label{eq:EigenvalueProof1}
            \int_E f(X,Y) d(X,Y),
        \end{equation}
        where $d(X,Y) = \prod_{1 \leq i \leq j \leq 2} dX_{ij}dY_{ij}$ is the Lebesgue measure on the space of pairs of $2 \times 2$ symmetric matrices and $E \colonequals \{(X,Y) \in S(\lambda) \times S(\mu) : (\lambda,\mu) \in D\}$. Let the function $p$ be defined by \eqref{eq:GaussianEigenvalues} on $\R^2_\geq \times \R^2_\geq$. Our goal is to show that the integral of $p$ over $D$ provides an upper bound on \eqref{eq:EigenvalueProof1}. By the spectral theorem, for each $X \in S(\lambda)$ we can find an orthogonal matrix $Q \in \R^{2 \times 2}$ such that $X = Q^T \diag(\lambda) Q$, where $\diag(\lambda)$ denotes the diagonal matrix with entries $\lambda_1,\lambda_2$. The columns of $Q$ are given by normalized eigenvectors $v_1,v_2$ corresponding to $\lambda_1,\lambda_2$. Since the set of matrices without distinct eigenvalues has Lebesgue measure zero, we may assume that $\lambda_1 \neq \lambda_2$, in which case $v_1,v_2$ are unique up to signs. To obtain a unique representation, let us assume that $v_1$ lies above the $x$-axis and $v_2$ lies to the right of the $y$-axis. Then $v_1 = (\cos\theta,\sin\theta)^T$ and $v_2 = (\sin\theta,-\cos\theta)^T$ for some $\theta \in [0,\pi)$. In particular, we have $X = Q^T \diag(\lambda) Q$ with
        \begin{equation*}
            Q =
            \begin{pmatrix}
                \cos\theta & \sin\theta \\
                \sin\theta & -\cos\theta
            \end{pmatrix}
        \end{equation*}
        for some $\theta \in [0,\pi)$. Writing out $Q^T \diag(\lambda) Q$ explicitly yields the parametrization
        \begin{equation}
            \label{eq:EigenvalueProof2}
            P(\lambda,\theta) =
            \begin{pmatrix}
                \lambda_1\cos^2\theta + \lambda_2\sin^2\theta & (\lambda_1 - \lambda_2)\sin\theta\cos\theta \\
                (\lambda_1 - \lambda_2)\sin\theta\cos\theta & \lambda_1\sin^2\theta + \lambda_2\cos^2\theta
            \end{pmatrix}.
        \end{equation}
        Then, we can write $E = \{(P(\lambda,\theta),P(\mu,\psi)) : (\lambda,\mu) \in D, \ \theta,\psi \in [0,\pi)\}$. The map $P : \R^2 \times [0,\pi) \to S(\lambda)$ defined by \eqref{eq:EigenvalueProof2} is differentiable and bijective according to the previous discussion. Since symmetric matrices are uniquely determined by their upper triangular entries, we can interpret $P$ as a map from $\R^3$ to $\R^3$. Thus, by a change of variables and Fubini's theorem, we can transform \eqref{eq:EigenvalueProof1} into
        \begin{equation}
            \label{eq:EigenvalueProof3}
            \int_D \int_{[0,\pi) \times [0,\pi)} f(P(\lambda,\theta),P(\mu,\psi)) \abs{\det(J(\lambda,\theta,\mu,\psi))} d(\theta,\psi) d(\lambda,\mu),
        \end{equation}
        where $J(\lambda,\theta,\mu,\psi)$ denotes the Jacobian of the product map $P \times P$ at the point $(\lambda,\theta,\mu,\psi)$. Its determinant can be evaluated to
        \begin{equation*}
            \abs{\det(J(\lambda,\theta,\mu,\psi))} = \abs{(\lambda_2 - \lambda_1)(\mu_2 - \mu_1)} = \Delta(\lambda)\Delta(\mu).
        \end{equation*}
        Since the eigenvalues are ordered, it follows from Neumann's trace inequality that
        \begin{equation*}
            \tr(P(\lambda,\theta)P(\mu,\psi)) \leq \lambda_1\mu_1 + \lambda_2\mu_2 = \tr(\diag(\lambda)\diag(\mu))
        \end{equation*}
        and \cref{lem:GaussianMatrix2} implies that
        \begin{equation*}
            f(P(\lambda,\theta),P(\mu,\psi)) \leq f(\diag(\lambda),\diag(\mu)).
        \end{equation*}
        Integrating this inequality over the domain $[0,\pi) \times [0,\pi)$ yields
        \begin{equation*}
            \eqref{eq:EigenvalueProof3} \leq \int_D \pi^2 f(\diag(\lambda),\diag(\mu)) \Delta(\lambda)\Delta(\mu) d(\lambda,\mu).
        \end{equation*}
        Noting that $\pi^2 K_2^2 = C_2^2$ and applying \cref{lem:GaussianMatrix2} to get
        \begin{equation*}
            \pi^2 f(\diag(\lambda),\diag(\mu)) \Delta(\lambda)\Delta(\mu) = p(\lambda,\mu)
        \end{equation*}
        completes the proof of the two-dimensional case.

        The proof of the general case follows along the same lines. By the spectral theorem, we can represent $X = Q^T \diag(\lambda) Q$ for some orthogonal matrix $Q \in \R^{m \times m}$. The set of all $m \times m$ orthogonal matrices $\mathcal{O}(m)$ equipped with matrix multiplication forms a Lie group. It is well known that the Lie algebra $\mathfrak{o}(m)$ of $\mathcal{O}(m)$ consists of the $m \times m$ skew-symmetric matrices. Furthermore, the exponential map from $\mathfrak{o}(m)$ to $\mathcal{O}(m)$ is surjective, that is, for each $Q \in \mathcal{O}(m)$ we can find a skew-symmetric matrix $A \in \R^{m \times m}$ such that $X = \exp{A} \diag(\lambda) \exp{-A}$. This provides a parametrization of $\mathcal{O}(m)$ using $m(m - 1)/2$ parameters. Note that every skew-symmetric matrix is uniquely determined by its upper-diagonal entries $(a_{ij})_{1 \leq i < j \leq n}$. However, making this parametrization bijective and differentiable requires some further technical details, which can be found in Section 2.5.2 of \cite{Anderson10}.
    \end{proof}

    \begin{proof}[Proof of \cref{lem:SmallNormProbability}]
        Since the spectral norm of a symmetric matrix corresponds to the maximum absolute value of its eigenvalues, it follows from \cref{lem:GaussianEigenvalues1} that
        \begin{equation}
            \label{eq:SmallNormProof1}
            \P{\norm{X} \leq \delta} = C_m \int_{D_\delta} \exp{-\frac{\norm{\lambda}_2^2}{4}} \Delta(\lambda) d\lambda,
        \end{equation}
        where $D_\delta \colonequals \{ \lambda \in \R^m : -\delta \leq \lambda_1 \leq \ldots \leq \lambda_m \leq \delta \}$ and $C_m$ is defined as in \eqref{eq:NormalizationConstant2}. Since $\delta = o(\sqrt{m})$, we can estimate $\exp{-\norm{\lambda}_2^2/4} \geq \exp{-o(1)m^2/4}$ for $\lambda \in D_\delta$. Therefore, it remains to bound the integral of $\Delta(\lambda)$ over $D_\delta$. By shifting and rescaling the integration variable $\lambda$ and using symmetry, we find that
        \begin{equation}
            \label{eq:SmallNormProof2}
            \int_{D_\delta} \Delta(\lambda) d\lambda = (2\delta)^{m(m-1)/2} \frac{1}{m!} \int_{[0,1]^m} \abs{\Delta(\lambda)} d\lambda.
        \end{equation}
        The latter integral is well known and can be evaluated in closed form. According to Selberg's integral formula, see Theorem 2.5.8 in \cite{Anderson10}, we have
        \begin{equation}
            \label{eq:SmallNormProof3}
            \frac{1}{m!} \int_{[0,1]^m} \abs{\Delta(\lambda)} d\lambda = \prod_{i=0}^{m-1} \frac{\Gamma((i + 2)/2)^2\Gamma((i + 1)/2)}{\Gamma((m + i + 3)/2)\Gamma(1/2)}.
        \end{equation}
        Denoting $a_i \colonequals (i + 2)/2$ and $b_i \colonequals (m + i + 3)/2$ for $i = 0,\ldots,m-1$ and using the fact $\Gamma(1/2) = \sqrt{\pi}$, we can rewrite \eqref{eq:SmallNormProof3} as
        \begin{equation}
            \label{eq:SmallNormProof4}
            \frac{1}{m!} \int_{[0,1]^m} \abs{\Delta(\lambda)} d\lambda = \pi^{-m/2} \prod_{i=1}^m \Gamma(i/2) \prod_{i=0}^{m-1} \frac{\Gamma(a_i)^2}{\Gamma(b_i)}.
        \end{equation}
        It remains to estimate the product $P \colonequals \prod_{i=0}^{m-1} \Gamma(a_i)^2/\Gamma(b_i)$. Using the following double inequality for the Gamma function
        \begin{equation}
            \label{eq:GammaFunctionBounds}
            \sqrt{2\pi} z^{z - 1/2} e^{-z} \leq \Gamma(z) \leq e^{1/12} \sqrt{2\pi} z^{z - 1/2} e^{-z}
        \end{equation}
        that holds for all $z \geq 1$, see Equation 5.6.1 in \cite{Olver10}, we obtain the estimate
        \begin{equation*}
            P \geq (\sqrt{2\pi}e^{-1/12})^m \prod_{i=0}^{m-1} \frac{a_i^{2a_i - 1}}{b_i^{b_i - 1/2}} \exp{b_i - 2a_i} \geq \pi^{m/2} \prod_{i=0}^{m-1} \frac{a_i^{2a_i - 1}}{b_i^{b_i - 1/2}} \exp{b_i - 2a_i}.
        \end{equation*}
        Note that the right-hand side can be rewritten as
        \begin{equation*}
            \pi^{m/2} \exp{\sum_{i=0}^{m-1} \log(a_i)(2a_i - 1) - \log(b_i)(b_i - 1/2) + b_i - 2a_i}.
        \end{equation*}
        We split the sum into three parts
        \begin{equation*}
            S_1 \colonequals \sum_{i=0}^{m-1} \log(a_i)(2a_i - 1), \quad S_2 \colonequals \sum_{i=0}^{m-1} \log(b_i)(b_i - 1/2), \quad S_3 \colonequals \sum_{i=0}^{m-1} b_i - 2a_i.
        \end{equation*}
        For $S_1$ a Riemann sum approximation yields
        \begin{equation*}
            S_1 \gtrsim \int_0^{m-1} \log((x + 2)/2)(1 + x) dx \approx \frac{1}{2}\log(m)m^2 - \frac{1}{2}\log(2)m^2 - \frac{1}{4}m^2.
        \end{equation*}
        Similarly, for $S_2$ a Riemann sum approximation yields
        \begin{equation*}
            S_2 \lesssim \int_0^{m} \log((m + x + 3)/2)(m + x + 2)/2 dx \lesssim \frac{3}{4}\log(m)m^2 + \frac{1}{4}\log(2)m^2 - \frac{3}{8}m^2.
        \end{equation*}
        Furthermore, a straightforward calculation gives
        \begin{equation*}
            S_3 = \sum_{i=0}^{m-1} (m - i - 1)/2 = m(m-1)/4 \approx \frac{1}{4}m^2.
        \end{equation*}
        In total, we have $S_1 - S_2 + S_3 \gtrsim -\frac{1}{4}\log(m)m^2 -\frac{3}{4}\log(2)m^2 + \frac{3}{8}m^2$, and therefore
        \begin{equation}
            \label{eq:SmallNormProof5}
            P \geq \pi^{m/2} \exp{S_1 - S_2 + S_3} \geq \pi^{m/2} \left(\frac{e^{3/4}}{2^{3/2}\sqrt{m}}\right)^{(1 + o(1))m^2/2}.
        \end{equation}
        Combining our bound in \eqref{eq:SmallNormProof5} with \eqref{eq:SmallNormProof4} and recalling \eqref{eq:SmallNormProof3}, we conclude that
        \begin{equation*}
            \int_{D_\delta} \Delta(\lambda) d\lambda \geq \prod_{i=1}^m \Gamma(i/2) \left(\frac{e^{3/4}}{\sqrt{2m}}\delta\right)^{(1 + o(1))m^2/2}.
        \end{equation*}
        Finally, when the above bound is applied to \eqref{eq:SmallNormProof1}, we obtain
        \begin{equation*}
            \P{\norm{X} \geq \delta} \geq C_m \prod_{i=1}^m \Gamma(i/2) \left(\frac{e^{3/4}}{\sqrt{2m}}\delta\right)^{(1 + o(1))m^2/2}\geq \left(\frac{e^{3/4}}{2\sqrt{m}}\delta\right)^{(1 + o(1))m^2/2}.
        \end{equation*}
        Using the lower bound in \eqref{eq:GammaFunctionBounds} and perfoming similar Riemann sum approximations, we get an upper bound of the same order and the assertion follows.
    \end{proof}

    \begin{proof}[Proof of \cref{lem:LaplaceMethod}]
        Fix an arbitrary $\varepsilon > 0$. By Taylor's theorem for any $\delta > 0$ and any $x \in [y - \delta,y + \delta]$, there exists $\xi_n \in [y - \delta,y + \delta]$ such that
        \begin{equation}
            \label{eq:LaplaceProof1}
            \varphi_n(x) = \varphi_n(y) + \varphi_n'(y)(x - y) + \frac{1}{2} \varphi_n''(\xi_n)(x - y)^2.
        \end{equation}
        From the assumption that $y$ is an interior point at which $\varphi_n$ attains a maximum, it follows that $\varphi_n'(y) = 0$. Thus, we can rewrite \eqref{eq:LaplaceProof1} as
        \begin{equation*}
            \varphi_n(x) = \varphi_n(y) + \frac{1}{2} \varphi_n''(\xi_n)(x - y)^2.
        \end{equation*}
        Since $\varphi_n''$ is equicontinuous at $y$, we can control $\abs{\varphi_n''(\xi_n) - \varphi_n''(y)}$ independent of $n$. In particular, we can find $\delta > 0$ such that
        \begin{equation*}
            \frac{(1 - \varepsilon)^2}{2} \varphi_n''(y)(x - y)^2 \leq \varphi_n(x) - \varphi_n(y) \leq \frac{(1 + \varepsilon)^2}{2} \varphi_n''(y)(x - y)^2
        \end{equation*}
        for all $x \in [y - \delta,y + \delta]$ and $n \in \N$. Using the negativity assumption $\varphi''(y) < 0$, we can rewrite the latter expression as
        \begin{equation}
            \label{eq:LaplaceProof2}
            -\frac{(1 - \varepsilon)^2}{2} \abs{\varphi_n''(y)}(x - y)^2 \leq \varphi_n(x) - \varphi_n(y) \leq -\frac{(1 + \varepsilon)^2}{2} \abs{\varphi_n''(y)}(x - y)^2.
        \end{equation}
        Now let us decompose the integral on the left-hand side of \eqref{eq:LaplaceMethod} into two parts
        \begin{equation*}
            \underbrace{\int_{a}^{y - \delta} \exp{n\varphi_n(x)} dx + \int_{y + \delta}^{b} \exp{n\varphi_n(x)} dx}_{\equalscolon I_1} + \underbrace{\int_{y - \delta}^{y + \delta} \exp{n\varphi_n(x)} dx}_{\equalscolon I_2}.
        \end{equation*}
        First, consider the integral $I_1$. Since $\varphi_n$ is assumed to be concave with a unique maximum at $y$, it attains its maximum over $[a,y - \delta] \cup [y + \delta,b]$ at one of the boundary points $y - \delta$ or $y + \delta$. Using the assumption $\varphi_n''(y) \leq c$ and the upper bound in \eqref{eq:LaplaceProof2} yields
        \begin{equation}
            \label{eq:LaplaceProof3}
            I_1 \leq (b - a) \exp{n\varphi_n(y) + n\frac{(1 + \varepsilon)^2}{2}c\delta^2}.
        \end{equation}
        Since $c < 0$ and $\exp{-n}\sqrt{n} = o(1)$, for $n$ large enough holds
        \begin{equation*}
            \eqref{eq:LaplaceProof3} \leq \varepsilon \sqrt{\frac{2\pi}{n\abs{\varphi_n''(y)}}} \exp{n\varphi_n(y)}.
        \end{equation*}
        On the other hand, we have $I_1 \geq 0$ as the integrand is positive. Next, we consider the integral $I_2$. From \eqref{eq:LaplaceProof2} follows the lower bound
        \begin{equation}
            \label{eq:LaplaceProof4}
            I_2 \geq \exp{n\varphi_n(y)} \int_{y - \delta}^{y + \delta} \exp{-n\frac{(1 + \varepsilon)^2}{2}\abs{\varphi_n''(y)}(x - y)^2} dx.
        \end{equation}
        By a change of variables \eqref{eq:LaplaceProof4} becomes
        \begin{equation}
            \label{eq:LaplaceProof5}
            \frac{1}{\sqrt{n}} \exp{n\varphi_n(y)} \int_{-\sqrt{n}\delta}^{\sqrt{n}\delta} \exp{-\frac{(1 + \varepsilon)^2}{2}\abs{\varphi_n''(y)}x^2} dx
        \end{equation}
        and for $n$ sufficiently large holds
        \begin{equation}
            \label{eq:LaplaceProof6}
            \eqref{eq:LaplaceProof5} \geq \frac{1}{\sqrt{n}} \exp{n\varphi_n(y)} (1 - \varepsilon) \int_{-\infty}^{\infty} \exp{-\frac{(1 + \varepsilon)^2}{2}\abs{\varphi_n''(y)}x^2} dx.
        \end{equation}
        The latter is a Gaussian integral and can be evaluated to
        \begin{equation*}
            \eqref{eq:LaplaceProof6} = \sqrt{\frac{2\pi}{n\abs{\varphi_n''(y)}}} \exp{n\varphi_n(y)} \frac{1 - \varepsilon}{1 + \varepsilon}.
        \end{equation*}
        Similarly, from \eqref{eq:LaplaceProof2} follows the upper bound
        \begin{equation}
            \label{eq:LaplaceProof7}
            I_2 \leq \exp{n\varphi_n(y)} \int_{y - \delta}^{y + \delta} \exp{-n\frac{(1 - \varepsilon)^2}{2}\abs{\varphi_n''(y)}(x - y)^2} dx.
        \end{equation}
        A change of variables and using nonnegativity of the integrand yields
        \begin{equation*}
            \eqref{eq:LaplaceProof7} \leq \sqrt{\frac{2\pi}{n\abs{\varphi_n''(y)}}} \exp{n\varphi_n(y)} \frac{1}{1 - \varepsilon}.
        \end{equation*}
        Combining all results shows that $I_1 + I_2$ equals the right-hand side of \eqref{eq:LaplaceMethod} up to a multiplicative factor in $[1 - 2\varepsilon,1 + 4\varepsilon]$ for $\varepsilon$ sufficiently small. Since $\varepsilon > 0$ was chosen arbitrarily, the claim follows.
    \end{proof}

    \begin{proof}[Proof of \cref{lem:WishartConcentration}]
        By definition, $W = GG^T$ for some $m \times r$ Gaussian matrix $G$. In particular, the $r$ non-zero eigenvalues of $W$ coincide with the squared singular values of $G$. Let $\sigma_{\min}(G)$ and $\sigma_{\max}(G)$ denote the smallest and largest singular values of $G$, respectively. Using the Sudakov-Fernique inequality, it can be shown that
        \begin{equation*}
            \sqrt{m} - \sqrt{r} - t \leq \sigma_{\min}(G) \leq \sigma_{\max}(G) \leq \sqrt{m} + \sqrt{r} + t
        \end{equation*}
        with probability at least $1 - 2\exp{-\Omega(t^2)}$, see Corollary 7.3.3 and Exercise 7.3.4 in \cite{Vershynin18} for details. By taking $t = \varepsilon\sqrt{m}$ for some $\varepsilon > 0$ and using the assumption $r \ll m$, we conclude that
        \begin{equation*}
            \sqrt{m}(1 - 2\varepsilon) \leq \sigma_{\min}(G) \leq \sigma_{\max}(G) \leq \sqrt{m}(1 + 2\varepsilon)
        \end{equation*}
        with probability at least $1 - \exp{-\Omega(m)}$. From the identity $\norm{W}_F^2 = \sum_{i=1}^m \lambda_i(W)^2$ follows that
        \begin{equation*}
            rm^2(1 - 2\varepsilon)^4 \leq r\sigma_{\min}(G)^4 \leq \norm{W}_F^2 \leq r\sigma_{\max}(G)^4 \leq rm^2(1 + 2\varepsilon)^4.
        \end{equation*}
        If $\varepsilon$ is sufficiently small, we have that $1/2 \leq (1 - 2\varepsilon)^4 \leq (1 + 2\varepsilon)^4 \leq 3/2$ and therefore
        \begin{equation*}
            \P{\abs{\frac{1}{rm^2}\norm{W}_F^2 - 1} \leq \frac{1}{2}} \geq 1 - 2\exp{-\Omega(m)}.
        \end{equation*}
        For the centered Wishart matrix $\overline{W}$, we can recycle these bounds using the relation $\lambda_i(\overline{W}) = \lambda_i(W) - r$ $\overline{W}$. The pertubation by $r$ is negligible when $r \ll m$.
    \end{proof}

    \begin{proof}[Proof of \cref{lem:WishartNorm}]
        We have to show that for sufficiently small $c > 0$ holds
        \begin{equation*}
            \E\exp{\frac{c}{\sqrt{r}}\abs{\inp{GG^T - rI_m}{Y}}} \leq 2
        \end{equation*}
        for all symmetric matrices $Y \in \R^{m \times m}$ with $\norm{Y}_F = 1$. Let $g_1,\ldots,g_r$ denote the columns of $G$. Substituting $G^TG = \sum_{i=1}^r g_ig_i^T$ gives
        \begin{align}
            \notag
            \E\exp{\frac{c}{\sqrt{r}}\inp{GG^T - rI_m}{Y}} &= \exp{-c\sqrt{r}\tr(Y)} \E\exp{\frac{c}{\sqrt{r}} \sum_{i=1}^r g_i^TYg_i}
            \intertext{and using independence of the columns yields}
            \label{eq:WishartProof1}
            &= \exp{-c\sqrt{r}\tr(Y)} \left(\E\exp{\frac{c}{\sqrt{r}} g^TYg}\right)^r
        \end{align}
        where $g$ denotes an $m$-dimensional Gaussian vector. We compute
        \begin{align*}
            \E\exp{\frac{c}{\sqrt{r}}g^TYg} &= \frac{1}{(2\pi)^{m/2}} \int_{\R^m} \exp{\frac{c}{\sqrt{r}}x^TYx - \frac{1}{2}x^Tx} \\ &= \frac{1}{(2\pi)^{m/2}} \int_{\R^m} \exp{-\frac{1}{2}x^T\left(I_m - \frac{2c}{\sqrt{r}}Y\right)x}
        \end{align*}
        and note that the above integrand matches the density \eqref{eq:MultivariateNormalDistribution} of an $m$-dimensional Gaussian vector with mean $\mu = 0$ and covariance matrix $\Sigma = (I_m - \frac{2c}{\sqrt{r}}Y)^{-1}$ except for the factor $\det(\Sigma)^{-1/2}$. Consequently, we have that
        \begin{equation*}
            \E\exp{\frac{c}{\sqrt{r}}g^TYg} = {\det}\of{I_m - \frac{2c}{\sqrt{r}}Y}^{-1/2}.
        \end{equation*}
        Using the formula $\det(I_m - \frac{2c}{\sqrt{r}}Y) = \prod_{i=1}^m \lambda_i(I_n - \frac{2c}{\sqrt{r}}Y) = \prod_{i=1}^m (1 - \frac{2c}{\sqrt{r}}\lambda_i(Y))$, we find that
        \begin{equation}
            \label{eq:WishartProof2}
            \eqref{eq:WishartProof1} = \exp{-c\sqrt{r}\tr(Y) - \frac{r}{2}\sum_{i=1}^m {\log}\of{1 - \frac{2c}{\sqrt{r}}\lambda_i(Y)}}.
        \end{equation}
        Due to the assumption $\norm{Y} \leq \norm{Y}_F = 1$, for $c \leq 1/4$ holds $\frac{2c}{\sqrt{r}}\lambda_i(Y) \leq 1/2$. Thus, we can apply the inequality $t^2 + t + \log(1-t) \geq 0$ for $0 \leq t \leq 1/2$ to obtain
        \begin{align*}
            \eqref{eq:WishartProof2} &\leq \exp{-c\sqrt{r}\tr(Y) + \frac{r}{2}\sum_{i=1}^m \frac{2c}{\sqrt{r}}\lambda_i(Y) + \frac{4c^2}{r}\lambda_i(Y)^2}
        \intertext{where the first two terms cancel each other out as $\tr(Y) = \sum_{i=1}^n \lambda_i(Y)$, leaving}
            &= \exp{2c^2\norm{Y}_F^2} = \exp{2c^2}.
        \end{align*}
        Similarly, we obtain that
        \begin{equation*}
            \E\exp{-\frac{c}{\sqrt{r}}\inp{GG^T - rI_m}{Y}} \leq \exp{2c^2}
        \end{equation*}
        and using the inequality $\exp{\abs{t}} \leq \frac{2}{3}(\exp{2t} + \exp{-2t})$ for $t \in \R$ provides the desired claim.
    \end{proof}

    \clearpage

    \addcontentsline{toc}{section}{References}
    \bibliographystyle{plain}
    \bibliography{main}

\begin{thebibliography}{10}

\bibitem{Abbe22}
Emmanuel Abbe, Shuangping Li, and Allan Sly.
\newblock Proof of the contiguity conjecture and lognormal limit for the symmetric perceptron.
\newblock In {\em 2021 {IEEE} 62nd {A}nnual {S}ymposium on {F}oundations of {C}omputer {S}cience---{FOCS} 2021}, pages 327--338. IEEE Computer Soc., Los Alamitos, CA, 2022.

\bibitem{Adamczak11}
R.~Adamczak, A.~E. Litvak, A.~Pajor, and N.~Tomczak-Jaegermann.
\newblock Restricted isometry property of matrices with independent columns and neighborly polytopes by random sampling.
\newblock {\em Constr. Approx.}, 34(1):61--88, 2011.

\bibitem{Alon16}
Noga Alon and Joel~H. Spencer.
\newblock {\em The probabilistic method}.
\newblock Wiley Series in Discrete Mathematics and Optimization. John Wiley \& Sons, Inc., Hoboken, NJ, fourth edition, 2016.

\bibitem{Altschuler23}
Dylan~J. Altschuler.
\newblock Critical window of the symmetric perceptron.
\newblock {\em Electron. J. Probab.}, 28:Paper No. 123, 28, 2023.

\bibitem{Altschuler24}
Dylan~J. Altschuler.
\newblock Zero-one laws for random feasibility problems, 2024.

\bibitem{Altschuler22}
Dylan~J. Altschuler and Jonathan Niles-Weed.
\newblock The discrepancy of random rectangular matrices.
\newblock {\em Random Structures Algorithms}, 60(4):551--593, 2022.

\bibitem{Alweiss21}
Ryan Alweiss, Yang~P. Liu, and Mehtaab Sawhney.
\newblock Discrepancy minimization via a self-balancing walk.
\newblock In {\em S{TOC} '21---{P}roceedings of the 53rd {A}nnual {ACM} {SIGACT} {S}ymposium on {T}heory of {C}omputing}, pages 14--20. ACM, New York, 2021.

\bibitem{Anderson10}
Greg~W. Anderson, Alice Guionnet, and Ofer Zeitouni.
\newblock {\em An introduction to random matrices}, volume 118 of {\em Cambridge Studies in Advanced Mathematics}.
\newblock Cambridge University Press, Cambridge, 2010.

\bibitem{Aubin19}
Benjamin Aubin, Will Perkins, and Lenka Zdeborov\'a.
\newblock Storage capacity in symmetric binary perceptrons.
\newblock {\em J. Phys. A}, 52(29):294003, 32, 2019.

\bibitem{Banaszczyk98}
Wojciech Banaszczyk.
\newblock Balancing vectors and {G}aussian measures of {$n$}-dimensional convex bodies.
\newblock {\em Random Structures Algorithms}, 12(4):351--360, 1998.

\bibitem{Banaszczyk12}
Wojciech Banaszczyk.
\newblock On series of signed vectors and their rearrangements.
\newblock {\em Random Structures Algorithms}, 40(3):301--316, 2012.

\bibitem{Bansal10}
Nikhil Bansal.
\newblock Constructive algorithms for discrepancy minimization.
\newblock In {\em 2010 {IEEE} 51st {A}nnual {S}ymposium on {F}oundations of {C}omputer {S}cience---{FOCS} 2010}, pages 3--10. IEEE Computer Soc., Los Alamitos, CA, 2010.

\bibitem{Bansal14}
Nikhil Bansal, Moses Charikar, Ravishankar Krishnaswamy, and Shi Li.
\newblock Better algorithms and hardness for broadcast scheduling via a discrepancy approach.
\newblock In {\em Proceedings of the {T}wenty-{F}ifth {A}nnual {ACM}-{SIAM} {S}ymposium on {D}iscrete {A}lgorithms}, pages 55--71. ACM, New York, 2014.

\bibitem{Bansal19b}
Nikhil Bansal, Daniel Dadush, and Shashwat Garg.
\newblock An algorithm for {K}oml\'os conjecture matching {B}anaszczyk's bound.
\newblock {\em SIAM J. Comput.}, 48(2):534--553, 2019.

\bibitem{Bansal19a}
Nikhil Bansal, Daniel Dadush, Shashwat Garg, and Shachar Lovett.
\newblock The {G}ram-{S}chmidt walk: a cure for the {B}anaszczyk blues.
\newblock {\em Theory Comput.}, 15:Paper No. 21, 27, 2019.

\bibitem{Bansal23}
Nikhil Bansal, Haotian Jiang, and Raghu Meka.
\newblock Resolving matrix {S}pencer conjecture up to poly-logarithmic rank.
\newblock In {\em S{TOC}'23---{P}roceedings of the 55th {A}nnual {ACM} {S}ymposium on {T}heory of {C}omputing}, pages 1814--1819. ACM, New York, 2023.

\bibitem{Bansal21}
Nikhil Bansal, Haotian Jiang, Raghu Meka, Sahil Singla, and Makrand Sinha.
\newblock Online discrepancy minimization for stochastic arrivals.
\newblock In {\em Proceedings of the 2021 {ACM}-{SIAM} {S}ymposium on {D}iscrete {A}lgorithms ({SODA})}, pages 2842--2861. [Society for Industrial and Applied Mathematics (SIAM)], Philadelphia, PA, 2021.

\bibitem{Bansal22b}
Nikhil Bansal, Haotian Jiang, Raghu Meka, Sahil Singla, and Makrand Sinha.
\newblock Prefix discrepancy, smoothed analysis, and combinatorial vector balancing.
\newblock In {\em 13th {I}nnovations in {T}heoretical {C}omputer {S}cience {C}onference}, volume 215 of {\em LIPIcs. Leibniz Int. Proc. Inform.}, pages Art. No. 13, 22. Schloss Dagstuhl. Leibniz-Zent. Inform., Wadern, 2022.

\bibitem{Bansal22c}
Nikhil Bansal, Haotian Jiang, Raghu Meka, Sahil Singla, and Makrand Sinha.
\newblock Smoothed analysis of the {K}oml\'os conjecture.
\newblock In {\em 49th {EATCS} {I}nternational {C}onference on {A}utomata, {L}anguages, and {P}rogramming}, volume 229 of {\em LIPIcs. Leibniz Int. Proc. Inform.}, pages Art. No. 14, 12. Schloss Dagstuhl. Leibniz-Zent. Inform., Wadern, 2022.

\bibitem{Bansal20b}
Nikhil Bansal, Haotian Jiang, Sahil Singla, and Makrand Sinha.
\newblock Online vector balancing and geometric discrepancy.
\newblock In {\em S{TOC} '20---{P}roceedings of the 52nd {A}nnual {ACM} {SIGACT} {S}ymposium on {T}heory of {C}omputing}, pages 1139--1152. ACM, New York, 2020.

\bibitem{Bansal22a}
Nikhil Bansal, Aditi Laddha, and Santosh Vempala.
\newblock A unified approach to discrepancy minimization.
\newblock In {\em Approximation, randomization, and combinatorial optimization. {A}lgorithms and techniques}, volume 245 of {\em LIPIcs. Leibniz Int. Proc. Inform.}, pages Art. No. 1, 22. Schloss Dagstuhl. Leibniz-Zent. Inform., Wadern, 2022.

\bibitem{Bansal20a}
Nikhil Bansal and Joel~H. Spencer.
\newblock On-line balancing of random inputs.
\newblock {\em Random Structures Algorithms}, 57(4):879--891, 2020.

\bibitem{Barany81}
Imre B\'ar\'any.
\newblock A vector-sum theorem and its application to improving flow shop guarantees.
\newblock {\em Math. Oper. Res.}, 6(3):445--452, 1981.

\bibitem{Beck81}
J\'ozsef Beck and Tibor Fiala.
\newblock ``{I}nteger-making''\ theorems.
\newblock {\em Discrete Appl. Math.}, 3(1):1--8, 1981.

\bibitem{BenArous97}
G\'{e}rard Ben~Arous and Alice Guionnet.
\newblock Large deviations for {W}igner's law and {V}oiculescu's non-commutative entropy.
\newblock {\em Probab. Theory Related Fields}, 108(4):517--542, 1997.

\bibitem{Boyd04}
Stephen Boyd and Lieven Vandenberghe.
\newblock {\em Convex optimization}.
\newblock Cambridge University Press, Cambridge, 2004.

\bibitem{Vempala14}
Karthekeyan Chandrasekaran and Santosh~S. Vempala.
\newblock Integer feasibility of random polytopes.
\newblock In {\em I{TCS}'14---{P}roceedings of the 2014 {C}onference on {I}nnovations in {T}heoretical {C}omputer {S}cience}, pages 449--458. ACM, New York, 2014.

\bibitem{Chazelle00}
Bernard Chazelle.
\newblock {\em The discrepancy method}.
\newblock Cambridge University Press, Cambridge, 2000.
\newblock Randomness and complexity.

\bibitem{Costello09}
Kevin~P. Costello.
\newblock Balancing {G}aussian vectors.
\newblock {\em Israel J. Math.}, 172:145--156, 2009.

\bibitem{Dadush16}
Daniel Dadush, Shashwat Garg, Shachar Lovett, and Aleksandar Nikolov.
\newblock Towards a constructive version of {B}anaszczyk's vector balancing theorem.
\newblock In {\em Approximation, randomization, and combinatorial optimization. {A}lgorithms and techniques}, volume~60 of {\em LIPIcs. Leibniz Int. Proc. Inform.}, pages Art. No. 28, 12. Schloss Dagstuhl. Leibniz-Zent. Inform., Wadern, 2016.

\bibitem{Dadush22}
Daniel Dadush, Haotian Jiang, and Victor Reis.
\newblock A new framework for matrix discrepancy: partial coloring bounds via mirror descent.
\newblock In {\em S{TOC} '22---{P}roceedings of the 54th {A}nnual {ACM} {SIGACT} {S}ymposium on {T}heory of {C}omputing}, pages 649--658. ACM, New York, 2022.

\bibitem{Bruijn81}
Nicolaas~Govert de~Bruijn.
\newblock {\em Asymptotic methods in analysis}.
\newblock Dover Publications, Inc., New York, third edition, 1981.

\bibitem{Eisenbrand18}
Friedrich Eisenbrand and Robert Weismantel.
\newblock Proximity results and faster algorithms for integer programming using the {S}teinitz lemma.
\newblock In {\em Proceedings of the {T}wenty-{N}inth {A}nnual {ACM}-{SIAM} {S}ymposium on {D}iscrete {A}lgorithms}, pages 808--816. SIAM, Philadelphia, PA, 2018.

\bibitem{Franks20}
Cole Franks and Michael Saks.
\newblock On the discrepancy of random matrices with many columns.
\newblock {\em Random Structures Algorithms}, 57(1):64--96, 2020.

\bibitem{Gamarnik23}
David Gamarnik, Eren~C. Kızıldağ, Will Perkins, and Changji Xu.
\newblock Geometric barriers for stable and online algorithms for discrepancy minimization, 2023.

\bibitem{Giannopoulos97}
Apostolos~A. Giannopoulos.
\newblock On some vector balancing problems.
\newblock {\em Studia Math.}, 122(3):225--234, 1997.

\bibitem{Gluskin88}
Efim~Davydovich Gluskin.
\newblock Extremal properties of orthogonal parallelepipeds and their applications to the geometry of {B}anach spaces.
\newblock {\em Mat. Sb. (N.S.)}, 136(178)(1):85--96, 1988.

\bibitem{Gopalan10}
Parikshit Gopalan, Ryan O'Donnell, Yi~Wu, and David Zuckerman.
\newblock Fooling functions of halfspaces under product distributions.
\newblock In {\em 25th {A}nnual {IEEE} {C}onference on {C}omputational {C}omplexity---{CCC} 2010}, pages 223--234. IEEE Computer Soc., Los Alamitos, CA, 2010.

\bibitem{Harshaw24}
Christopher Harshaw, Fredrik S\"avje, Daniel~A. Spielman, and Peng Zhang.
\newblock Balancing covariates in randomized experiments with the {G}ram-{S}chmidt walk design.
\newblock {\em J. Amer. Statist. Assoc.}, 119(548):2934--2946, 2024.

\bibitem{Hoberg17}
Rebecca Hoberg and Thomas Rothvoss.
\newblock A logarithmic additive integrality gap for bin packing.
\newblock In {\em Proceedings of the {T}wenty-{E}ighth {A}nnual {ACM}-{SIAM} {S}ymposium on {D}iscrete {A}lgorithms}, pages 2616--2625. SIAM, Philadelphia, PA, 2017.

\bibitem{Hoberg19}
Rebecca Hoberg and Thomas Rothvoss.
\newblock A {F}ourier-analytic approach for the discrepancy of random set systems.
\newblock In {\em Proceedings of the {T}hirtieth {A}nnual {ACM}-{SIAM} {S}ymposium on {D}iscrete {A}lgorithms}, pages 2547--2556. SIAM, Philadelphia, PA, 2019.

\bibitem{Hopkins22}
Samuel~B. Hopkins, Prasad Raghavendra, and Abhishek Shetty.
\newblock Matrix discrepancy from quantum communication.
\newblock In {\em S{TOC} '22---{P}roceedings of the 54th {A}nnual {ACM} {SIGACT} {S}ymposium on {T}heory of {C}omputing}, pages 637--648. ACM, New York, 2022.

\bibitem{Jansen19}
Klaus Jansen and Lars Rohwedder.
\newblock On integer programming and convolution.
\newblock In {\em 10th {I}nnovations in {T}heoretical {C}omputer {S}cience}, volume 124 of {\em LIPIcs. Leibniz Int. Proc. Inform.}, pages Art. No. 43, 17. Schloss Dagstuhl. Leibniz-Zent. Inform., Wadern, 2019.

\bibitem{Karmarkar86}
Narendra Karmarkar, Richard~M. Karp, George~S. Lueker, and Andrew~M. Odlyzko.
\newblock Probabilistic analysis of optimum partitioning.
\newblock {\em J. Appl. Probab.}, 23(3):626--645, 1986.

\bibitem{Kulkarni24}
Janardhan Kulkarni, Victor Reis, and Thomas Rothvoss.
\newblock Optimal online discrepancy minimization.
\newblock In {\em S{TOC}'24---{P}roceedings of the 56th {A}nnual {ACM} {S}ymposium on {T}heory of {C}omputing}, pages 1832--1840. ACM, New York, 2024.

\bibitem{Latala05}
Rafa{\l} Lata{\l}a.
\newblock Some estimates of norms of random matrices.
\newblock {\em Proc. Amer. Math. Soc.}, 133(5):1273--1282, 2005.

\bibitem{Liu22}
Yang~P. Liu, Ashwin Sah, and Mehtaab Sawhney.
\newblock A {G}aussian fixed point random walk.
\newblock In {\em 13th {I}nnovations in {T}heoretical {C}omputer {S}cience {C}onference}, volume 215 of {\em LIPIcs. Leibniz Int. Proc. Inform.}, pages Art. No. 101, 10. Schloss Dagstuhl. Leibniz-Zent. Inform., Wadern, 2022.

\bibitem{Lovett10}
Shachar Lovett.
\newblock An elementary proof of anti-concentration of polynomials in {Gaussian} variables.
\newblock In {\em Electron. Colloquium Comput. Complex.}, volume~17, page 182, 2010.

\bibitem{Lovett15}
Shachar Lovett and Raghu Meka.
\newblock Constructive discrepancy minimization by walking on the edges.
\newblock {\em SIAM J. Comput.}, 44(5):1573--1582, 2015.

\bibitem{Lust91}
Fran\c{c}oise Lust-Piquard and Gilles Pisier.
\newblock Noncommutative {K}hintchine and {P}aley inequalities.
\newblock {\em Ark. Mat.}, 29(2):241--260, 1991.

\bibitem{Maillard24}
Antoine Maillard.
\newblock Average-case matrix discrepancy: satisfiability bounds, 2024.

\bibitem{Marcus15}
Adam~W. Marcus, Daniel~A. Spielman, and Nikhil Srivastava.
\newblock Interlacing families {II}: {M}ixed characteristic polynomials and the {K}adison-{S}inger problem.
\newblock {\em Ann. of Math. (2)}, 182(1):327--350, 2015.

\bibitem{Matousek10}
Ji\v{r}\'{i} Matou\v{s}ek.
\newblock {\em Geometric discrepancy}, volume~18 of {\em Algorithms and Combinatorics}.
\newblock Springer-Verlag, Berlin, 2010.
\newblock An illustrated guide, Revised paperback reprint of the 1999 original.

\bibitem{Meckes06}
Elizabeth Meckes.
\newblock {\em An infinitesimal version of Stein’s method of exchangeable pairs}.
\newblock PhD thesis, Stanford University, 2006.

\bibitem{Nikolov17}
Aleksandar Nikolov.
\newblock Tighter bounds for the discrepancy of boxes and polytopes.
\newblock {\em Mathematika}, 63(3):1091--1113, 2017.

\bibitem{Olver10}
Frank W.~J. Olver, Daniel~W. Lozier, Ronald~F. Boisvert, and Charles~W. Clark, editors.
\newblock {\em N{IST} handbook of mathematical functions}.
\newblock U.S. Department of Commerce, National Institute of Standards and Technology, Washington, DC; Cambridge University Press, Cambridge, 2010.

\bibitem{Paley32}
Raymond Paley and Antoni Zygmund.
\newblock A note on analytic functions in the unit circle.
\newblock {\em Mathematical Proceedings of the Cambridge Philosophical Society}, 28(3), 1932.

\bibitem{Perkins21}
Will Perkins and Changji Xu.
\newblock Frozen 1-{RSB} structure of the symmetric {I}sing perceptron.
\newblock In {\em S{TOC} '21---{P}roceedings of the 53rd {A}nnual {ACM} {SIGACT} {S}ymposium on {T}heory of {C}omputing}, pages 1579--1588. ACM, New York, 2021.

\bibitem{Potukuchi18}
Aditya Potukuchi.
\newblock Discrepancy in random hypergraph models, 2018.
\newblock arXiv:1811.01491.

\bibitem{Reis20}
Victor Reis and Thomas Rothvoss.
\newblock Linear size sparsifier and the geometry of the operator norm ball.
\newblock In {\em Proceedings of the 2020 {ACM}-{SIAM} {S}ymposium on {D}iscrete {A}lgorithms}, pages 2337--2348. SIAM, Philadelphia, PA, 2020.

\bibitem{Rothvoss17}
Thomas Rothvoss.
\newblock Constructive discrepancy minimization for convex sets.
\newblock {\em SIAM J. Comput.}, 46(1):224--234, 2017.

\bibitem{Sah23}
Ashwin Sah and Mehtaab Sawhney.
\newblock Distribution of the threshold for the symmetric perceptron.
\newblock In {\em 2023 {IEEE} 64th {A}nnual {S}ymposium on {F}oundations of {C}omputer {S}cience---{FOCS} 2023}, pages 2369--2382. IEEE Computer Soc., Los Alamitos, CA, 2023.

\bibitem{Sevastjanov94}
Sergey~Vasil'evich Sevast'janov.
\newblock On some geometric methods in scheduling theory: a survey.
\newblock {\em Discrete Appl. Math.}, 55(1):59--82, 1994.

\bibitem{Soshnikov99}
Alexander Soshnikov.
\newblock Universality at the edge of the spectrum in {W}igner random matrices.
\newblock {\em Comm. Math. Phys.}, 207(3):697--733, 1999.

\bibitem{Spencer77}
Joel Spencer.
\newblock Balancing games.
\newblock {\em J. Combinatorial Theory Ser. B}, 23(1):68--74, 1977.

\bibitem{Spencer85}
Joel Spencer.
\newblock Six standard deviations suffice.
\newblock {\em Trans. Amer. Math. Soc.}, 289(2):679--706, 1985.

\bibitem{Spencer94}
Joel Spencer.
\newblock {\em Ten lectures on the probabilistic method}, volume~64 of {\em CBMS-NSF Regional Conference Series in Applied Mathematics}.
\newblock Society for Industrial and Applied Mathematics (SIAM), Philadelphia, PA, second edition, 1994.

\bibitem{Tao12}
Terence Tao.
\newblock {\em Topics in random matrix theory}, volume 132 of {\em Graduate Studies in Mathematics}.
\newblock American Mathematical Society, Providence, RI, 2012.

\bibitem{Tropp12}
Joel~A. Tropp.
\newblock User-friendly tail bounds for sums of random matrices.
\newblock {\em Found. Comput. Math.}, 12(4):389--434, 2012.

\bibitem{Tsuda05}
Koji Tsuda, Gunnar R\"atsch, and Manfred~K. Warmuth.
\newblock Matrix exponentiated gradient updates for on-line learning and {B}regman projection.
\newblock {\em J. Mach. Learn. Res.}, 6:995--1018, 2005.

\bibitem{Turner20}
Paxton Turner, Raghu Meka, and Philippe Rigollet.
\newblock Balancing gaussian vectors in high dimension.
\newblock In Jacob Abernethy and Shivani Agarwal, editors, {\em Proceedings of Thirty Third Conference on Learning Theory}, volume 125 of {\em Proceedings of Machine Learning Research}, pages 3455--3486. PMLR, 2020.

\bibitem{Vershynin18}
Roman Vershynin.
\newblock {\em High-dimensional probability}, volume~47 of {\em Cambridge Series in Statistical and Probabilistic Mathematics}.
\newblock Cambridge University Press, Cambridge, 2018.

\bibitem{Zouzias12}
Anastasios Zouzias.
\newblock A matrix hyperbolic cosine algorithm and applications.
\newblock In {\em Automata, languages, and programming. {P}art {I}}, volume 7391 of {\em Lecture Notes in Comput. Sci.}, pages 846--858. Springer, Heidelberg, 2012.

\end{thebibliography}
\end{document}